\documentclass{article} 
\usepackage[left=1.5cm, right=1.5cm, top=1.5cm, bottom = 1.5cm]{geometry}

\usepackage[binary-units=true]{siunitx}
\usepackage{times}
\usepackage{amsmath,amssymb,amsthm}

\usepackage{accents}
\usepackage{graphicx}
\usepackage{bm}

\usepackage[inline]{enumitem}
\usepackage{array}
\usepackage{multicol}
\usepackage{multirow}
\usepackage{tabulary}
\usepackage{booktabs}
\usepackage{subfig}
\usepackage{todonotes}
\usepackage{cite}
\usepackage{physics}
\usepackage[hidelinks]{hyperref}
\usepackage{algorithm}
\usepackage{algpseudocode}
\usepackage{breqn}
\usepackage{tcolorbox}
\usepackage{mathtools}
\usepackage{cuted}
\usepackage{stfloats}
\usepackage{balance}
\usepackage[toc,page]{appendix}
\usepackage{caption}

\usepackage{tcolorbox}
\usepackage{mathtools}
\usepackage{tkz-euclide}
\usepackage{nomencl}
\usepackage{pgfplots}
\usepackage{pifont}

\newcommand{\xmark}{\ding{55}}%

\allowdisplaybreaks

\DeclareMathOperator{\diag}{diag}

\DeclareMathOperator*{\mini}{min.}
\DeclareMathOperator*{\lexmin}{lexmin.}

\definecolor{light-gray}{gray}{0.95}
\definecolor{dark-gray}{gray}{0.5}
\definecolor{mygray}{gray}{0.75}
\newcommand{\BIN}{\begin{bmatrix}}
\newcommand{\BOUT}{\end{bmatrix}}

\newcommand{\mA}{{\mathcal{A}}}

\newcommand{\mI}{{\mathcal{I}}}

\newcommand{\bE}{{\mathbb{E}}}
\newcommand{\bI}{{\mathbb{I}}}
\newcommand{\bC}{{\mathbb{C}}}

\newcommand{\pluseq}{\mathrel{+}=}

\pgfdeclarelayer{bg1}   
\pgfdeclarelayer{bg}  
\pgfsetlayers{bg1,bg,main}  

\definecolor{orange}{rgb}{0.99,0.69,0.07}
\definecolor{lightgray}{gray}{0.85}
\definecolor{light-gray}{gray}{0.95}
\definecolor{dark-gray}{gray}{0.5}

\tikzset{cross/.style={cross out, draw=black, minimum size=2*(#1-\pgflinewidth), inner sep=0pt, outer sep=0pt},
cross/.default={1pt}}

 \makeatletter
 \newcommand\fs@spaceruled{\def\@fs@cfont{\bfseries}\let\@fs@capt\floatc@ruled
   \def\@fs@pre{\vspace{5pt}\hrule height.8pt depth0pt \kern2pt}%
   \def\@fs@post{\kern2pt\hrule\relax}%
   \def\@fs@mid{\kern2pt\hrule\kern2pt}%
   \let\@fs@iftopcapt\iftrue}
 \makeatother

\newtheorem{theorem}{Theorem}

\makeatletter
\newcommand{\thickhline}{%
	\noalign {\ifnum 0=`}\fi \hrule height 1pt
	\futurelet \reserved@a \@xhline
}
\newcolumntype{"}{@{\hskip\tabcolsep\vrule width 1pt\hskip\tabcolsep}}
\makeatother

\title{\LARGE \bf Dual Hierarchical Least-Squares Programming\\ with Equality Constraints}
\author{Kai Pfeiffer\thanks{Department of Mechatronics and Robotics, School of Advanced Technology, Xi'an Jiaotong Liverpool University, Suzhou, China; The author is supported by the RDF-24-02-060 XJTLU Research Development Fund.
	}%
}

\date{}
\begin{document}
	
	\maketitle
	\thispagestyle{empty}
	\pagestyle{empty}
	
	\begin{abstract}%
		Hierarchical least-squares programming (HLSP) is an important tool in optimization as it enables the stacking of any number of priority levels in order to reflect complex constraint relationships, for example in physical systems like robots. Existing solvers typically address the primal formulation of HLSP's, which is computationally efficient due to sequential treatment of the priority levels. This way, already identified active constraints can be eliminated after each priority level, leading to smaller problems as the solver progresses through the hierarchy. However, this sequential progression makes the solvers discontinuous and therefore not differentiable. This prevents the incorporation of HLSP's as neural network neurons, or solving HLSP's in a distributed fashion. In this work, an efficient solver based on the dual formulation of HLSP's with equality constraints (\ref{eq:d-hlsp}) is developed. \ref{eq:d-hlsp} is a convex and differentiable quadratically constrained least-squares program (QCLSP), which is solved by an Alternating Direction Method of Multipliers (ADMM). By introducing appropriate operator splitting, primal-dual variables, which link each priority level with all their respective higher priority levels, can be eliminated from the main computation step of computing a matrix factorization. The proposed solver D-HADM is about one magnitude faster than a comparable \ref{eq:d-hlsp} solver based on the interior-point method (D-HIPM), where  the  primal-dual variables enter the computational complexity in a cubic fashion.
	\end{abstract}
	
	\section{Introduction}
	Hierarchical least-squares programming (HLSP) has seen many applications, for example in compliant control of humanoid robots \cite{Herzog2016}, in control of manipulators with obstacle avoidance in space stations~\cite{Xing2025} or for robot assisted surgeries~\cite{Colan2024}, in robot design optimization~\cite{Ginnante2023}, or in control of robot swarms for object transportation \cite{Koung2021}. The ability to prioritize constraints by any number of priority levels accommodates simple formulation of complex constraint relationships, since constraints do not need to be weighted against each other, for example in complex scenarios with many constraints like in humanoid stair climbing  \cite{Vaillant2016}. Efficient solvers based on the primal formulation of HLSP's have been proposed, which solve a sequence of least-squares programs (LSP) from highest to lowest priority level~\cite{escande2014}. This is computationally advantageous since after convergence of each priority level, active constraints and occupied variables can be eliminated by nullspace projections and do not further need to be  considered on lower hierarchy levels. However, the overall  solution of the HLSP is not differentiable, as it is  the sum of independent contributions from the sequentially solved LSP's. The work in \cite{Sathya2021} presents a dual formulation of hierarchical linear programs (HLP), which results in a single linear program with linear constraints. However, an efficient solver has not been proposed. Critically, the dual HLP grows with each priority level in the number of primal variables, since the dual variables from previous levels need to be considered as primal ones on subsequent levels.
	
	Given these restrictions on HLSP, this work presents several advances for HLSP's with equality constraints (HLSP-E):
	\begin{itemize}
		\itemsep0pt
		\item The dual formulation of HLSP-E's as a quadratically constrained least-squares program (QCLSP, a sub-form of quadratically constrained quadratic programs / QCQP; see Sec. \ref{sec:d-hlsp}) is derived, which can be solved with methods from convex optimization. This would allow to solve HLSP-E's in a distributed fashion~\cite{Chen2021}, for example for distributed prioritized robot control.
		\item The gradient of the dual formulation of the HLSP-E is formulated using techniques from matrix differential calculus (see App. \ref{sec:diffdhlsp}). This lays the foundation for future applications like inclusion of HLSP-E's as neural network neurons for representing hierarchical relationships in the training data. 
		\item An efficient solver for the dual HLSP-E based on the Alternating Direction Method of Multipliers (ADMM) is developed (Sec. \ref{sec:hadmm}). Appropriate operator splitting ensures that primal-dual variables do not need to be considered in demanding matrix factorizations. It is demonstrated how such an efficient formulation cannot be immediately found for an interior-point method (IPM) based solver (see App. \ref{app:ipm}).
	\end{itemize}
	This article is structured as follows. First, Sec.~\ref{sec:relatedwork} gives an overview of existing literature on hierarchical programming.
	Section~\ref{sec:backgroundhlsp} deepens this insight by introducing the primal HLSP \eqref{eq:hlsp}. In contrast, Sec.~\ref{sec:d-hlsp} formulates the dual HLSP for equality constraints \ref{eq:d-hlsp}, which is a convex QCLSP. Such programs can be solved by methods from convex optimization, where this work introduces exemplary solvers both based on the ADMM (Sec. \ref{sec:hadmm}) and the IPM (App. \ref{app:ipm}). Section~\ref{sec:hpcompare} categorizes these solvers with respect to existing solvers in hierarchical programming with linear constraints.
	Evaluation on randomized and linear deficient test programs takes place in Sec.~\ref{sec:eval}. Section~\ref{sec:conclusion} concludes this article with a future outlook and reiterates on the advantages and disadvantages of dual HLSP.

	\section{Related work}
	\label{sec:relatedwork}
	HLSP can be broadly categorized as lexicographic multi-objective optimization (LMOO), which has its origin in \cite{Sherali1983}. This work establishes appropriate weights which transform a hierarchy of linear objectives subject to a set of linear equality constraints into an equivalently weighted optimization problem.
	Inequality constraints are limited to non-negativity constraints on the variable vector.
	The authors in \cite{Kanoun2011} are the first to establish HLSP in its current form, which is characterized by the ability to handle feasible and infeasible inequality constraints on any priority level. This is enabled by slack variables which relax any infeasibility with respect to the $\ell_2$-norm. Each priority level is solved in sequence by identifying the optimal active set by the active-set method (ASM). The method is computationally unfavorable since on each priority level an increasingly larger problem needs to be solved since all active and inactive constraints from previous levels need to be considered. 
	The authors in \cite{escande2014} avoid this by projecting the constraints of the remaining priority levels into a basis of nullspace of the active constraints of previous priority levels. This effectively eliminates the active contraints, and with the right choice of nullspace basis also eliminates variables which have already been occupied in order to reach the optimal point of previous priority levels. The ASM is very efficient for problems with small changes of the active-set, or when the problem is well-conditioned, for example in parametric programs like robot control and planning where the active-set typically evolves slowly~\cite{Kuindersma2014}. Here, the solver can be warm-started and the solver converges after a few iterations, which involve solving equality constrained problem representing the active constraints of the current priority level. However, when the problem is badly conditioned, and large changes of the active set occur, the ASM can struggle to converge, for example due to phenomenons like cycling~\cite{Gill1989}. Instead, the work in \cite{pfeiffer2021} proposes a HLSP solver based on the interior-point method (IPM), which is more reliable than the ASM in such situations. The solver cannot be warm-started and each iteration is more costly than those of the ASM since all inactive inequality constraints need to be considered.
	In contrast, the solver in \cite{pfeiffer2024b} is based on the ADMM, which can be warm-started and can handle ill-posed problems due to proximal regularization terms~\cite{Parikh2014}. The ADMM requires fewer expensive matrix factorizations. This is more suitable for embedded applications where the system may not be equipped with enough computational power and matrix factorizations can be computed offline \cite{dang2015}.
	
	The previously mentioned works consider the primal formulation of hierarchical programs, while the dual has rarely been considered. One work to do so is presented in \cite{Sathya2021}, where the authors formulate the dual of a HLP. The dual is again a constrained linear program, where all priority levels can be solved at once, unlike solvers considering the primal formulation which solve the program in sequence. However, the dual HLP increases in the number of primal variables, which now also include the dual variables of each priority level of the primal problem (from here on these variables are referred to as \textit{primal-dual} variables, as they are dual variables in the primal problem, but primal variables in the dual problem).  Furthermore, the gradient of the dual HLP has not been formulated, which is necessary if HLP's were to be integrated as neurons in neural networks. The inclusion of convex optimization programs like quadratic programs (QP) in neural networks has been demonstrated to benefit the representation of more complex constraint relationships in the training data \cite{optnet}. 
	
	Hierarchical programming is not limited to HLP's or HLSP's: the work in \cite{pfeiffer2025} addresses hierarchical sparse quadratic programs (HSQP) where the objective function is formulated with respect to the $\ell_0$-norm. This way, hierarchical decision-making problems can be solved, for example in goal selection for humanoid robots.
	
	HLSP's and HSQP's are typically solved as sub-problems of either non-linear hierarchical least-squares or sparse quadratic programs (N-HLSP, or N-HSQP) in sequential programming type non-linear optimization methods. Examples include trust-region methods \cite{pfeiffer2023} with fixed trust-region radius, and hierarchical adaptations \cite{pfeiffer2024} of filter methods for sequential quadratic programming (SQP) \cite{fletcher2002b}.
	Unlike previous methods from LMOO for non-linear optimization~\cite{He2018}, non-linear feasible and infeasible inequality constraints are handled seamlessly by the slack formulation.
	Examples of applications for such non-linear problems are robot real-time control \cite{pfeiffer2023} and prioritized robot optimal control \cite{pfeiffer2024}.

	\section{Background on Hierarchical Least-Squares Programming}
	\label{sec:backgroundhlsp}
	This work addresses HLSP as LMOO's ($\lexmin$: lexicograhpically minimize) 
	\begin{align}
	\lexmin_{{x},{v}_{\bC_{\cup p}}}& \qquad \frac{1}{2}\Vert {v}_{\bC_1} \Vert^2_{\ell=2},\dots,\frac{1}{2}\Vert {v}_{\bC_p} \Vert^2_{\ell=2},  \label{eq:lexhlsp}\tag{Lex-HLSP}\\
	\text{s.t.}
	& \qquad A_{\bC_{\cup p}}{x} - b_{\bC_{\cup p}} \leqq {v}_{\bC_{\cup p}}\nonumber
	\end{align}
	$x\in\mathbb{R}^{n_x}$ are the primal variables. On each level $l=1,\dots,p$, the primal slack $v_{\bC_l}\in\mathbb{R}^{m_{\bC_l}}$ needs to be minimized with respect to the $\ell_2$-norm while the optimal slacks of previous levels $1,\dots,l-1$ needs to be maintained. 
	The slacks relax the linear constraints $\bC_l = \{\bE_l,\bI_l\}$ of each level $l=1,\dots,p$, which are comprised of the set of equality constraints $\bE_l$, and the set of inequality constraints $\bI_l$. The symbol $\leqq$ summarily represents both these constraint sets. The linear constraints $\bC_l$ are represented by the matrices $A_{\bC_l}\in\mathbb{R}^{m_{\bC_l}\times n_x}$ and the vectors $b_{\bC_l}\in\mathbb{R}^{m_{\bC_l}}$. The symbol $\cup l \coloneqq 1,\dots,p$ indicates the set of integers representing the priority levels from level 1 to $p$. The constraint set $\bC_{\cup p}$ is then the combined set $\bC_{\cup p}\coloneqq \{\bC_{1},\dots,\bC_{p}\}$. 
	
	To this date, no method has been conceptualized which can solve all priority levels of \ref{eq:lexhlsp} at once. Instead, the priority levels are treated in sequence from highest to lowest priority level, resulting in the following primal (P-) HLSP
	\begin{align}
	\mini_{ x,{v}_{\mathbb{C}_l}}& \qquad \frac{1}{2}\Vert {v}_{\mathbb{C}_l} \Vert^2_2 \qquad l=1,\dots,p \label{eq:hlsp}\tag{P-HLSP}\\
	\text{s.t.}
	& \qquad A_{\mathbb{C}_l} x - b_{\mathbb{C}_l} \leqq {v}_{\mathbb{C}_l}\nonumber\\
	& \qquad A_{\mA_{\cup l-1}} x - b_{\mA_{\cup l-1}} = {v}_{\mA_{\cup l-1}}^*\nonumber\\
	& \qquad A_{\mI_{\cup l-1}} x - b_{\mI_{\cup l-1}} \leq 0\nonumber
	\end{align}
	Each level $l=1,\dots,p$ is a feasible LSP. Either, the set of linear constraints $\bC_l$ are feasible and an optimal feasible point $v_{\bC_l}^* \leqq 0$ is identified. Otherwise, the linear constraints $\bC_l$ are infeasible, and are relaxed at an optimal infeasible point $\vert v_{\bC_l}^* \vert > 0$ in the least-squares sense. At the same time, the optimal feasible or infeasible points $\vert v_{\mA_{\cup l-1}}^*\vert \geq 0$ of higher priority levels $1$ to $l-1$ need to be maintained. The active set $\mathcal{A}_{\cup l-1}$ contains all equality constraints $\bE_{\cup l-1}$, and violated inequality constraints in $\bI_{\cup l-1}$. The remaining inequality constraints from higher priority levels, which were feasible at convergence of the LSP's from levels 1 to $l-1$, are summarized in $\mathcal{I}_{\cup l-1}$ and need to be respected. 
	
	Such a sequential treatment of the hierarchy is very efficient, since after convergence of each priority level $l$, identified active constraints from this level can be eliminated from the remaining priority levels. The approach in~\cite{escande2014} proposed to project the remaining levels' LSP's into a basis of nullspace $N_{l-1}$ of the active constraints of lower priority levels $\mathcal{A}_{\cup l-1}$ with $A_{\mathcal{A}_{\cup l-1}}N_{l-1} = 0$ and $N_{0}=I\in\mathbb{R}^{n_x\times n_x}$. An appropriate choice of nullspace basis $N_{l-1}\in\mathbb{R}^{n_x\times n_{x}^{r,l}}$ carries the additional advantage that the LSP's become progressively smaller with a reduced number of remaining projected variables $n_{x}^{r,l} \leq n_x$ on each level $l=2,\dots,p$. The rank of the active constraint matrix $A_{\mathcal{A}_{\cup l-1}}$ is $n_x - n_{x}^{r,l}$. The projected primal variable on each level is $x_l \in \mathbb{R}^{ n_{x}^{r,l}} = N_lx$. Furthermore, the projection eliminates the primal-dual Lagrange multipliers $\lambda_{l,\cup l-1}\in\mathbb{R}^{n_l^{dual}}$ from each level $l$'s LSP, which relate the LSP to linear constraints $\bC_{\cup l-1}$ from previous levels $1$ to $l-1$. 
	$n_{l}^{dual}$ is the sum of number of constraints of priority levels 1 to $l-1$:
	\begin{align}
	n_{l}^{dual} = \sum_{k=1}^{l-1} m_{\bC_l} \nonumber
	\end{align} 
	The overall set of primal-dual variables is defined as \cite{escande2014}
	\begin{align}
	\Lambda_{p-1} = \BIN 0 & \lambda_{2,\bC_1} & \dots & \lambda_{p-1,\bC_1} \\
	0 & 0 & \dots &  \lambda_{p-1,\bC_2}\\
	\vdots & \vdots & \ddots & \vdots\\
	0 & 0 & \dots &  \lambda_{p-1,\bC_{p-2}}\\
	\label{eq:primaldual}\tag{primal-dual}
	\BOUT \in\mathbb{R}^{n_{p-1}^{dual} \times p-1}
	\end{align}
	The non-zero entries of each column $l=2,\dots,p-1$ of $\Lambda_{p-1}$ correspond to $\lambda_{l,\cup l-1}$. The matrix $\Lambda_{p-1}$ omits the Lagrange multipliers $\lambda_{l,l}$. They correspond to the slack variables $\lambda_{l,\bC_l} = v_{\bC_l}$ and are simultaneously both primal and dual variables \cite{escande2014}.
	
	The primal and primal-dual solutions $x^*$ and $\Lambda_{p-1}^*$ obtained from solving the primal formulation of \ref{eq:hlsp} are not differentiable with respect to the HLSP parameters $A_{\bC_{\cup p}}$ and $b_{\bC_{\cup p}}$. This is because it is not possible to formulate a Lagrangian function  of all priority levels' LSP's at once. This prevents the application to distributed optimization or the usage of HLSP's as neurons in neural networks. Here, we propose a continuous and differentiable dual formulation of HLSP.

	\section{Dual formulation of hierarchical least-squares programming}
	\label{sec:d-hlsp}
	
	In the following, the dual formulation of \ref{eq:lexhlsp} is derived. It can be seen that this results in a QCLSP as stated in \cite{Sathya2021}, which is differentiable, as is demonstrated in App. \ref{sec:diffdhlsp}.  
	
	Following Lagrangian function is assumed for some level $p$ of a dual HLSP:
	\begin{align}
	\mathcal{L}_{p}({x},v_{\bC_{\cup p}},\lambda_{p,\bC_{\cup p}},\Lambda_{p-1},\eta_{p,\cup p-1},\theta_{p,\cup p-1}) &= \frac{1}{2}\Vert {v}_{\bC_p} \Vert^2_2 + \sum_{l=1}^{p} \mu_{p,\bC_l}^T(A_{\bC_l}{x} - b_{\bC_l} - {v}_{\bC_l} + {w}_{\bC_l}) \label{eq:d-hlsp-lag}\\
	&{\color{blue}+\sum_{l=1}^{p-1}  \eta_{p,l}^T\left(A_{\bC_l}^Tv_{\bC_l}  +  A_{\bC_{\cup l-1}}^T\lambda_{l,\bC_{\cup l-1}}\right)} \nonumber\\
	&{\color{blue}+ \sum_{l=1}^{p-1}\theta_{p,l} \left({v}_{\bC_l}^T({v}_{\bC_l} + b_{\bC_l}) +  \lambda_{l,\bC_{\cup l-1}}^Tb_{\bC_{\cup l-1}} + \sum_{k=1}^{l-1}\theta_{l,k}v_{\bC_k}^Tv_{\bC_k}\right)} \nonumber
	\end{align}
	$\mu_{p,\bC_{\cup p}}\in\mathbb{R}^{n_{p}^{dual}}$ are the dual variables of level $p$ with respect to the linear constraints $\bC_{\cup p}$ of levels 1 to $p$.
	Components specific to the dual HLSP formulation are marked in blue and their origin becomes apparent throughout this derivation. Importantly, the primal-dual variables $\lambda_{l,\bC_{\cup l-1}}\in\mathbb{R}^{n_{l}^{dual}}$  link internally each priority level $l=2,\dots,p-1$ with their respective higher priority ones $1,\dots,l-1$. In the primal HLSP, on each level's LSP they are the dual variables $\mu_{l,\bC_{\cup l-1}}$ with respect to their respective higher priority levels, but in the dual HLSP, they are treated as primal variables $\lambda_{l,\bC_{\cup l-1}}$ on any priority level except the lowest priority level $p$. $\eta_{p,l}$ and $\theta_{p,l}$ with $l=1,\dots, l-1$ are Lagrange multipliers with respect to constraints specifically arising from the dual formulation. Similarly to the Lagrange multipliers $\lambda_{l,\bC_{\cup l-1}}$, the variables $\theta_{l,k}$ are first dual variables on the levels $l=2,\dots,p-1$ with $k=1,\dots,l-1$, but later primal-dual variables on level $p$. The following Lagrange mulitpliers are linked to inequality constraints ($\leq$) and are subject to non-negativity conditions:
	\begin{itemize}
		\itemsep0pt
		\item $\mu_{\bI_{\cup p}}\geq 0$
		\item $\lambda_{l,\bI_{\cup l-1}}\geq 0$ (with $l=2,\dots,p-1$)
		\item $\theta_{l,k} \geq 0$ (with $l=2,\dots,p$, $k=1,\dots,l-1$)
	\end{itemize}
	The first-order optimality (Karush-Kuhn-Tucker, KKT) conditions of the Lagrangian \eqref{eq:d-hlsp-lag} with respect to the primal variables are
	\begin{align}
	\nabla_{{x}} \mathcal{L}_p =& \sum_{l=1}^{p} A_{\bC_l}^T\mu_{p,\bC_l} = 0\label{eq:kkt1}\\
	\nabla_{v_{\bC_p}} \mathcal{L} =& v_{\bC_p} - \mu_{p,\bC_p} = 0 \label{eq:lpIsVp}\\
	\nabla_{v_{\bC_l}} \mathcal{L}_p =& -\mu_{p,\bC_l} + \theta_{p,l}(2{v}_{\bC_l} + b_{\bC_l}) + A_{\bC_l}\eta_{l} + 2\theta_{p,l}\sum_{k=1}^{l-1}\theta_{l,k}v_{\bC_k}^Tv_{\bC_k}= 0\qquad\qquad l = 1,\dots,p-1\\
	\nabla_{\theta_{l,k}} \mathcal{L}_p =& \theta_{p,l}v_{\bC_k}^Tv_{\bC_k}= 0 \qquad\qquad l = 2,\dots,p-1,\qquad k =1,\dots,l-1 \label{eq:4thorderelim}\\
	\nabla_{\lambda_{l,\bC_{\cup l-1}}} \mathcal{L}_p =& \theta_{p,l}b_{\bC_{\cup l-1}} + A_{\bC_{\cup l-1}}\eta_{p,l} = 0 \qquad\qquad l = 1,\dots,p-1
	\end{align}
	The condition \eqref{eq:lpIsVp} demonstrates how the variables $v_{\bC_p} = \mu_{p,\bC_p}$ are both primal and dual variables at the same time. 
	The KKT conditions with respect to the dual variables are
	\begin{align}
	\nabla_{\mu_{p,\bC_{\cup p-1}}} \mathcal{L}_p =& A_{\bC_{\cup p-1}}x - b_{\bC_{\cup p-1}} - v_{\bC_{\cup p-1}}=0\label{eq:kkt2}\\
	\nabla_{\eta_{p,l}} \mathcal{L}_p =& A_{\bC_l}^Tv_{\bC_l}  +  A_{\bC_{\cup l-1}}^T\lambda_{l,\bC_{\cup l-1}}=0\qquad l = 1\dots,p-1\label{eq:keta}\\
	\nabla_{\theta_{p,l}} \mathcal{L}_p =& {v}_{\bC_l}^T({v}_{\bC_l} + b_{\bC_l}) +  \lambda_{l,\bC_{\cup l-1}}^Tb_{\bC_{\cup l-1}} + \sum_{k=1}^{l-1}\theta_{l,k}v_{\bC_k}^Tv_{\bC_k}= 0 \qquad l = 1,\dots p-1\label{eq:thirorderterm}
	\end{align}
	Notably, the fourth-order coupling $2\theta_{p,l}\sum_{k=1}^{l-1}\theta_{l,k}v_{\bC_k}^Tv_{\bC_k}$ does not further occur, since it has been eliminated by condition \eqref{eq:4thorderelim}. 
	
	Furthermore, non-negativity  and complementary constraints on the dual and primal-dual variables associated with inequality constraints $\bI$ need to be respected
	\begin{align}
	&\mu_{p,\bI_{\cup p}} \geq 0 \qquad \text{and} \qquad \mu_{p,\bI_{\cup p}}\odot (A_{\bI_{\cup p}}x - b_{\bI_{\cup p}} - v_{\bI_{\cup p}})\nonumber
	\end{align} 
	The primal KKT conditions are inserted into $\mathcal{L}_p$, which results in the dual function of level $p$
	\begin{align}
	g_p({v}_{\mathbb{C}_{\cup p}},\lambda_{p,\bC_{\cup p-1}}) &=  -\frac{1}{2}\Vert {v}_{\mathbb{C}_p} \Vert^2_2  - {v}_{\mathbb{C}_p}^Tb_{\mathbb{C}_p} -  \lambda_{p,\bC_{\cup p-1}}^T b_{\bC_{\cup p-1}} -\sum_{l=1}^{p-1} \theta_{p,l} {v}_{\bC_{l}}^T{v}_{\bC_{l}}
	\end{align}

	The dual variables $\mu_{p,\bC_{\cup p-1}}$ of level $p$ now become the primal variables $\lambda_{p,\bC_{\cup p-1}}$ of the dual problem (similarly for $\theta_{p,l}$, $l=1,\dots,p-1$). The dual problem is given in the following, while the optimality condition \eqref{eq:keta} has been introduced as an explicit constraint \cite{boyd2004}
	\begin{align}
	\max. \quad&  g_p({v}_{\mathbb{C}_{\cup p}},\lambda_{p,\bC_{\cup p-1}}) \label{eq:dualproblem}\\
	\text{s.t} \quad& 
	A_{\bC_l}^Tv_{\bC_l}  +  A_{\bC_{\cup l-1}}^T\lambda_{l,\bC_{\cup l-1}} = 0\qquad l = 1,\dots,p-1\nonumber\\
	&{v}_{\bC_l}^T({v}_{\bC_l} + b_{\bC_l}) +  \lambda_{l,\bC_{\cup l-1}}^Tb_{\bC_{\cup l-1}} + \sum_{k=1}^{l-1}\theta_{l,k}v_{\bC_k}^Tv_{\bC_k} \leq 0\qquad l = 1,\dots,p-1\nonumber\\
	&\theta_{l,k} \geq 0\qquad l = 1,\dots,p,\qquad k = 1,\dots,l-1\nonumber\\
	& {\lambda_{p,\bI_{\cup p-1}}\odot ( A_{\bI_{\cup p-1}}{x} - b_{\bC_{\cup p-1}} - {v}_{\bI_{\cup p-1}}) = 0 }\nonumber\\
	&\lambda_{p,\bI_{\cup p-1}} \geq 0 \nonumber
	\end{align}
	The primal objective is lower bounded by the dual objective $\frac{1}{2} v_{\bC_p}^{*,T}v_{\bC_p}^{*} \geq g_p({v}_{\mathbb{C}_{\cup p}}^{*},\lambda_{p,\bC_{\cup p-1}}^{*})$ \cite{boyd2004}
	\begin{align}
	v_{\bC_p}^{*,T}v_{\bC_p}^{*} 	  + {v}_{\mathbb{C}_p}^{*,T}b_{\mathbb{C}_p} +  \lambda_{p,\bC_{\cup p-1}}^{*,T} b_{\bC_{\cup p-1}} +\sum_{l=1}^{p-1} \theta_{p,l}^* {v}_{\bC_{l}}^{*,T}{v}_{\bC_{l}}^{*}\geq& 0	\label{eq:weakduality}
	\end{align}
	It can be observed that strong duality (strict equality in \eqref{eq:weakduality}) holds if the constraints $\bC_l$ are feasible with $v_{\bC_l}^*=0$ and $\lambda_{l,\bC_{\cup l-1}}^*=0$, which follows from \eqref{eq:keta}.
	After some algebraic reformulations and using the relationship $\hat{b}_{\mathbb{C}_l}= \frac{1}{2}b_{\mathbb{C}_l}$, this can be rewritten to standard quadratic form
	\begin{align}
	& \qquad ({v}_{\mathbb{C}_l}^{*} + \hat{b}_{\mathbb{C}_l})^T({v}_{\mathbb{C}_l}^{*} + \hat{b}_{\mathbb{C}_l}) - \hat{b}_{\mathbb{C}_l}^T\hat{b}_{\mathbb{C}_l}  +  \lambda_{l,\bC_{\cup l-1}}^{*,T} b_{\bC_{\cup l-1}} { + \sum_{k=1}^{l-1}\theta_{l,k}^{*} {v}_{\bC_{k}}^{*,T}{v}_{\bC_{k}}^*}  \geq 0
	\label{eq:weakduality_quad}
	\end{align}
	From this formulation it becomes apparent that the reverse of the duality condition \eqref{eq:weakduality} is a convex quadratic constraint. The authors in \cite{Sathya2021} proposed to  transfer this constraint and the constraints from the dual problem~\eqref{eq:dualproblem} to the primal problem of the next priority level $p\leftarrow p+1$. This enforces zero-duality gap and therefore optimiality of previous priority levels 1 to $p-1$. The resulting dual HLSP \eqref{eq:d-hlsp} writes as
	\begin{align}
	\mini_{{x},{v}_{\bC_{\cup p}},\Lambda_{ p-1},{\color{cyan}\theta_{2,1},\dots,\theta_{p-1,p-2}}}& \qquad \frac{1}{2}\Vert {v}_{\bC_p} \Vert^2_2  \label{eq:d-hlsp}\tag{D-HLSP-E}\\
	\text{s.t.}
	\qquad&  A_{\bC_{\cup p}}{x} - b_{\bC_{\cup p}} \leqq {v}_{\bC_{\cup p}} \\
	& A_{\bC_l}^T{v}_{\bC_l} + A_{\bC_{\cup l-1}}^T\lambda_{l,\bC_{\cup l-1}}= 0 \qquad l = 1,\dots,p-1\label{eq:ATL}\\
	&  ({v}_{\mathbb{C}_l} + \hat{b}_{\mathbb{C}_l})^T({v}_{\mathbb{C}_l} + \hat{b}_{\mathbb{C}_l}) - \hat{b}_{\mathbb{C}_l}^T\hat{b}_{\mathbb{C}_l}  + \lambda_{l,\bC_{\cup l-1}}^T b_{\bC_{\cup l-1}} {\color{cyan} +\sum_{k=1}^{l-1} \theta_{l,k} {v}_{\bC_{k}}^T{v}_{\bC_{k}}}  \leq 0 \qquad l = 1,\dots,p-1\label{eq:qc}\\
	&{\color{cyan}\theta_{l,k} \geq 0\qquad l = 2,\dots,p-1,\qquad k = 1,\dots,l-1}\\
	& {\color{purple}\lambda_{p-1,\bI_{\cup p-2}}\odot ( A_{\bI_{\cup p-2}}{x} - b_{\bC_{\cup p-2}} - {v}_{\bI_{\cup p-2}}) = 0}\label{eq:Kcomplementary}\\
	& {\color{violet}\lambda_{p-1,\bI_{\cup p-2}} \geq 0 }\label{eq:Knonneg}
	\end{align}
	Notably, for any $l>2$ and with inequality constraints on any level 1 to $l-2$, a non-convexity is introduced by the complementary condition \eqref{eq:Kcomplementary} on the primal-dual variables $\lambda_{l-1,\bI_{\cup l-2}}$ (colored in purple). In this work, the focus is shifted to equality only constrained HLSP's, such that both \eqref{eq:Kcomplementary} and the non-negativity constraint \eqref{eq:Knonneg} are not further considered.
	
	Another non-convexity is introduced by the duality constraint \eqref{eq:qc}, which includes a positive third-order term $\sum_{k=1}^{l-1}\theta_{l,k} {v}_{\bC_{k}}^T{v}_{\bC_{k}}$ (in purple, $l=1,\dots,p-1$, $k=1,\dots,l-1$). However, it can be observed that the multiplier $\theta_{l,k}$ is uncoupled since it does not appear in any other constraint or the objective function. The optimality condition~\eqref{eq:4thorderelim} contains the only occurrence of $\theta_{l,k}$, and $\theta_{l,k}$ can therefore be chosen freely. Given non-negativity of the term $\sum_{k=1}^{l-1}\theta_{l,k} {v}_{\bC_{k}}^T{v}_{\bC_{k}}\geq 0$ with $\theta_{l,k}\geq 0$, the choice $\theta_{l,k}=0$ favors strong duality (strict equality in \eqref{eq:weakduality}), while $\theta_{l,k}>0$ would increase the duality gap. 
	
	With the above simplifications, \ref{eq:d-hlsp} becomes a convex QCLSP as all colored terms are ignored. 
	It can be confirmed that the Lagrangian given in~\eqref{eq:d-hlsp-lag} is the Lagrangian of~\ref{eq:d-hlsp}, where $\eta_{p,l}$ and $\theta_{p,l}$ with $l = 1,\dots,p-1$ are the Lagrange multipliers associated with the constraints in \eqref{eq:ATL} and \eqref{eq:qc}.
	
	Finally, it remains to be seen whether the duality constraint \eqref{eq:qc} is feasible. In the following, $x$ is constructed with sub-variables $z_{l=1,\dots,p}$ projected into a basis of nullspace of constraint matrices of previous priority levels.
	\begin{align}
	x = Z_0(z_1 + Z_1(z_2 + Z_2(z_3 + \dots))) = \sum_{l=1}^p N_{\cup{l-1}}z_l
	\end{align}
	with $Z_{0} = I$, $A_{\bE_l}N_{\cup l-1}Z_{l} = 0$ and $N_{\cup l} = \Pi_{k=1}^l Z_{k}$ with $l=1,\dots,p$.
\begin{theorem}
	Strong duality holds for the equality constrained HLSP \eqref{eq:d-hlsp}. 
\end{theorem}
\begin{proof}
	The proof proceeds by establishing strong duality from highest to lowest priority levels $l=1,\dots,p$, effectively solving the primal HLSP \eqref{eq:hlsp}. The  first priority level writes as
	\begin{align}
	\mini_{{z_1},{v}_{\mathbb{E}_1}}& \qquad \frac{1}{2}\Vert {v}_{\mathbb{E}_1} \Vert^2_2  \label{eq:l1}\\
	\text{s.t.}
	& \qquad A_{\mathbb{E}_1}{z_1} - b_{\mathbb{E}_1} = {v}_{\mathbb{E}_1}\nonumber
	\end{align}
	Problem~\eqref{eq:l1} is a QP with affine equality constraints, which are feasible due to relaxation by the slack variable $v_{\bE_1}$.
	This is a sufficient condition for strong duality~\cite{boyd2004}.
	\begin{align}
	{v}_{\mathbb{E}_1}^{*,T}({v}_{\mathbb{E}_1}^* + b_{\mathbb{E}_1})=0\nonumber
	\end{align} 
	Furthermore, $A_{\mathbb{E}_1}^T{v}_{\mathbb{E}_1}^* = 0$~\eqref{eq:kkt1} holds at optimality. The optimal point $(z_1^*, v_{\bE_1}^*)$ is kept constant hereinafter.
	\ref{eq:d-hlsp} of the second level in dependence of $z_2$ and $v_{\bE_2}$ writes as
	\begin{align}
	\mini_{z_2,{v}_{\mathbb{E}_2}}& \qquad \frac{1}{2}\Vert {v}_{\mathbb{E}_2} \Vert^2_2  \nonumber\\
	\text{s.t.}
	& \qquad A_{\mathbb{E}_1}z^*_1 - b_{\mathbb{E}_1} = {v}_{\mathbb{E}_1}^*\\
	& \qquad A_{\mathbb{E}_2}(z_1^* + N_{\bE_1}z_2) - b_{\mathbb{E}_2} = {v}_{\mathbb{E}_2}\nonumber\\
	& \qquad {v}_{\mathbb{E}_1}^{*,T}({v}_{\mathbb{E}_1}^* + b_{\mathbb{E}_1})  = 0 \nonumber\\
	& \qquad A_{\mathbb{E}_1}^T{v}_{\mathbb{E}_1}^* = 0 \quad \nonumber
	\end{align}
	The quadratic equality constraint is constant and can therefore be ignored. 
	Feasibility of the  affine equality constraints follows from primal feasibility of the first level, and relaxation by the slack variables $v_{\bE_2}$. This is a sufficient condition for strong duality of level~2:
	\begin{align}
	& \qquad {v}_{\mathbb{E}_2}^{*,T} ({v}_{\mathbb{E}_2}^* + b_{\mathbb{E}_2}) + \lambda_{2,\mathbb{E}_1}^{*,T}b_{\mathbb{E}_1} + \theta_{2,1}^* {v}_{\mathbb{E}_1}^{*,T}{v}_{\mathbb{E}_1}^*  = 0\nonumber
	\end{align}
	 The primal-dual variable / Lagrange multiplier $\lambda^*_{2,\bE_1}\leftarrow\mu^*_{2,\bE_1}$result from the KKT condition with respect to $x$~\eqref{eq:kkt1} (note that while $z_1^*$ is constant, $x$ itself is not).
	 $\theta_{2,1}^*\geq0$ is a free variable and has been chosen appropriately such that the above equality holds. 
	 Following this pattern for all priority levels 1 to $p$, strong duality
	\begin{align}
	v_{\bC_l}^{*,T}(v_{\bC_l}^{*} 	  + b_{\mathbb{C}_l}) +  \lambda_{l,\bC_{\cup l-1}}^{*,T} b_{\bC_{\cup l-1}} +\sum_{k=1}^{l-1} \theta_{l,k}^* {v}_{\bC_{k}}^{*,T}{v}_{\bC_{k}}^* = 0\nonumber \qquad l=1,\dots,p-1
	\end{align}
	follows for \ref{eq:d-hlsp}.
\end{proof}

	\section{An efficient ADMM for solving \ref{eq:d-hlsp}}
	\label{sec:hadmm}
	
	\begin{algorithm}[h!]
		\caption{\tt ADMM based solver for \ref{eq:d-hlsp}}
		\begin{algorithmic}[1]
			\Statex \textbf{Input:} $A_{\bE_{\cup p}}$, $b_{\bE_{\cup p}}$, $x^{(0)}$, $v_{\bE_{\cup p}}^{(0)}$ $\Lambda_{p-1}^{(0)}$, $\eta_{p,{\cup p-1}}^{(0)}$, $\phi_{p,{\cup p-1}}^{(0)}$, $\nu_{l,\cup l-1}^{(0)}$ (with $l = 2,\dots,p-1$)
			\Statex \textbf{Output:} $x^*$, $\Lambda^*$ 
			\Statex\textbf{Variables:} $i=0$
			\State Precondition $A_{\bE_{\cup p}}$, $b_{\bE_{\cup p}}$ (Sec.~\ref{sec:precond})
			\State Factorize $K_x^{\rho}$~\eqref{eq:Kx}
			\State Compute $M_l^{-1}$ for $l = 1, \dots,p-2$ (Sec.~\ref{sec:subpd})
			\While{$\Vert k_{prim,dual} \Vert_2 \geq \chi$}
			\State Update $x^{(i+1)}$ by solving \eqref{eq:Kx}
			\State Update $v_{\bE_{\cup p-1}}^{(i+1)}$ (Sec.~\ref{sec:subpd}), $v_{\bE_p}^{(i+1)}$~\eqref{eq:subvp}
			\State Update $\Lambda_{p-1}^{(i+1)}$ \eqref{eq:comppd} 
			\State Update $z_{\cup p-1}^{(i+1)}$, $\tilde{\lambda}$ by solving \ref{eq:qcproj}
			\State Update dual Lagrange multipliers (Sec.~\ref{sec:dualascent})
			\State Update $k_{prim,dual}^{(i+1)}$~\eqref{eq:kprimdual}
			\State Update $\rho^{(i+1)}$ (Sec.~\ref{sec:updaterho})
			\If{$\rho^{(i+1)} > 5\rho^{(i)}$ or $\rho^{(i+1)} < \frac{1}{5}\rho^{(i)}$}
			\State Re-factorize $K_x^{\rho}$~\eqref{eq:Kx}
			\EndIf
			\State $i \leftarrow i + 1$
			\EndWhile 
			\State Reverse equilibration of the variables \eqref{eq:equVar}
		\end{algorithmic}
		\label{alg:hadm}
	\end{algorithm}
	In App.~\ref{app:ipm}, it is demonstrated that using the IPM to solve \ref{eq:d-hlsp} leads to a Newton system where the primal-dual Lagrange multipliers $\Lambda_{p-1}$ cannot be eliminated. This is expensive as in each solver iteration, a square matrix of dimension $n + \sum_{l=2}^{p-1} n_{l}^{dual}$ needs to be factorized. It appears to be critical for computational efficiency that this is avoided as the computational effort of matrix factorizations grows cubically in the number of matrix dimensions. 
	In the following, it is demonstrated how this can be avoided by designing an ADMM based solver where the resulting substituted KKT system is only dependent on the primal variables $x$. 
	Algorithm Alg. \ref{alg:hadm} summarized following steps of the solver. First, Sec. \ref{sec:splitvars} introduces different split variables both to cast \ref{eq:d-hlsp} into ADMM format, but also to enhance computational efficiency. This is exploited in the computation of the primal update Sec.~\ref{sec:primupdate} where the primal-dual variables are substituted in order to avoid their occurrence in expensive matrix factorizations. In Sec.~\ref{sec:updatesplit}, it is described how the quadratic constraint \eqref{eq:qc} is treated in a computationally efficient way, where the main computational step consists of solving a third-order polynomial.
	The choice of partial pre-conditioning preserves this computational advantage, as otherwise a more expensive but factorization-free IPM solver (see App.~\eqref{app:ipmepsilon}) would need to be employed (Sec. \ref{sec:precond}).
	Sections \ref{sec:dualascent}, \ref{sec:updaterho} and \ref{sec:overrelax} recall ADMM specific elements for the dual and parameter updates, and over-relaxed steps, respectively.
	
	\subsection{Split variables for computational efficiency}
	\label{sec:splitvars}
	The ADMM is an optimization method which solves convex optimization problems of the form \cite{boyd2011}
	\begin{align}
	\mini_{q,s} \qquad &f(q) + d(s)\\
	& Bq + Cs = c\nonumber\\
	& s\in \mathcal{C}\nonumber
	\end{align}
	Both $f(q)$ and $d(s)$ are convex functions
	The variable $s$ is referred to as the split variable, which splits the original variable $q$ into two parts. The split variable $s$ is contained within the convex and closed constraint set $\mathcal{C}$ \cite{osqp}. In this work, three different splitting variables are introduced as follows
	\begin{itemize}
		\item $z_{\bE_{\cup p-1}}$: This split variable encapsulates the quadratic constraints~\eqref{eq:qc} of each level $l=1,
		\dots,p-1$, which now write as
		\begin{align}
		& \qquad z_{\bE_l}^Tz_{\bE_l} - \hat{b}_{\bE_l}^T\hat{b}_{\bE_l}  +  \tilde{\lambda}_{l,\bE_{\cup l-1}}^T b_{\bE_{\cup l-1}}  \leq 0\nonumber
		\end{align}
		Note that unlike~\cite{Krupa2024,Nguyen2025}, this constraint contains an additional linear term in $\tilde{\lambda}_{l,\bE_{\cup l-1}}$, yet adhering to the standard format of quadratic constraints in convex QCQP's~\cite{Anstreicher2012}.
		This operator splitting introduces the additional linear constraints
		\begin{align}
		z_{\bE_{\bE_{\cup p-1}}} = {v}_{\bE_{\bE_{\cup p-1}}} + \hat{b}_{\bE_{\bE_{\cup p-1}}}\nonumber
		\end{align}
		\item $\tilde{x}$: This split variable renders the KKT system positive-definite / invertible \cite{osqp}. The linear constraint $x = \tilde{x}$ is introduced.
		\item $\tilde{\lambda}_{l,\bE_{\cup l-1}}$ with $l=2,
		\dots,p-1$: Similarly to $\tilde{x}$, this split variable acts as regularization. It decouples the primal-dual variables ${\lambda}_{l,\bE_{\cup l-1}}$ of each level $l$ from the ones of the higher priority levels ${\lambda}_{k,\bE_{\cup k-1}}$ with $k=1,\dots,l-1$. This provides the basis for the elimination of the primal-dual variables from the substituted linear KKT system for computational efficiency.
		This introduces the additional linear constraints
		\begin{align}
		{\lambda}_{l,\bE_{\cup l-1}} = \tilde{\lambda}_{l,\bE_{\cup l-1}}\nonumber
		\end{align}
	\end{itemize} 
	The original program with introduced split variables and constraints now writes as
	\begin{align}
	\mini_{{x},{v}_{\bE_p}}\qquad&  \frac{1}{2}\Vert {v}_{\bE_p} \Vert^2_2 \\
	\text{s.t.} \qquad
	&  A_{\bE_l}{x} - b_{\bE_l} - {v}_{\bE_l} = \BIN 0 & -{w}_{\bI_l^T} \BOUT^T \qquad l=1,\dots,p\label{eq:lincstr}\\
	& A_{\bE_l}^Tv_{\bE_l} + A_{\bE_{\cup l-1}}^T\lambda_{l,\bE_{\cup l-1}} = 0 \qquad l=1,\dots,p-1 \label{eq:ATL2}\\
	&  {v}_{\bE_l}  + \hat{b}_{\bE_l}  = z_{\bE_l} \qquad l=1,\dots,p-1 \label{eq:z_lincstr}\\
	& {\lambda}_{l,\bE_{\cup l-1}} = \tilde{\lambda}_{l,\bE_{\cup l-1}} \qquad l=2,\dots,p-1\label{eq:tlam_lincstr}\\
	&z_{\bE_l},\tilde{\lambda}_{l,\bE_{\cup l-1}}\in C_{\epsilon},\qquad l = 1,\dots,p-1\nonumber\\
	& x=\tilde{x}\nonumber
	\end{align}
	The closed and convex set $ \mathcal{C}_{\epsilon}$ is defined as follows:
	\begin{align}
	&\mathcal{C}_{\epsilon} = \{z_{\bE_l},\tilde{\lambda}_{l,\bE_{\cup l-1}} \hspace{3pt}|\hspace{3pt} z_{\bE_l}^Tz_{\bE_l} - \hat{b}_{\bE_l}^T\hat{b}_{\bE_l} +  \tilde{\lambda}_{l,\bE_{\cup l-1}}^Tb_{\bE_{\cup {l-1}}}  \leq 0 \}\qquad l = 1,\dots,p-1\nonumber
	\end{align}
	The program is summarized as
	\begin{align}
	\mini_x \qquad& \frac{1}{2}q^THq + \sum_{l=1}^{p-1}I_{\epsilon}(z_{\bE_{l}},\tilde{\lambda}_{l,\bE_{\cup l-1}})
	\label{eq:admmdhlsp}\\
	\text{s.t.}\qquad& B_{\mu}q + C_{\mu}s = c_{\mu}\nonumber\\
	& B_{\eta}q + C_{\eta}s = c_{\eta}\nonumber\\
	& B_{\phi}q + C_{\phi}s = c_{\phi}\nonumber\\
	& B_{\nu}q + C_{\nu}s = c_{\nu}\nonumber\\
	& x = \tilde{x} \nonumber
	\end{align}
	The variable vectors are
	\begin{align}
	q &\coloneqq 
	\BIN x^T & v_{\bE_{\cup p}}^T & \dots & \lambda_{l,\bE_{\cup l-1}}^T &   \dots \BOUT^T\nonumber\\
	s&\coloneqq
	\BIN   z_{\bE_{\cup p-1}}^T &\dots & \tilde{\lambda}_{l,\bE_{\cup l-1}}^T & \dots \BOUT^T\nonumber
	\end{align}
	For $B$, $C$, and $c$, a distinction between the constraint sets 
	$\mu$, $\eta$, $\phi$, and $\nu$ is made. $\mu$ represents the constraints associated with the primal linear constraints~\eqref{eq:lincstr}, $\eta$ the dual linear constraints \eqref{eq:ATL2}, and $\phi$ and $\nu$ the representation of the quadratic dual constraints \eqref{eq:z_lincstr} and \eqref{eq:tlam_lincstr}, respectively. In each case, the same symbol is also used for the Lagrange multipliers associated with the corresponding constraint. The matrices are then defined as follows:
	\begin{align}
	H &= \diag(0,I_{m_{\bE_p}\times m_{\bE_p}},0)\nonumber
	\\
	B&\coloneqq \BIN B_{\mu}^T & B_{\eta}^T & B_{\phi}^T & B_{\nu}^T\BOUT^T\nonumber\\
	B_{\mu} &= \BIN A_{\bE_{\cup p}} & -I &  \dots & 0 & \dots \BOUT \nonumber\\
	B_{\eta} &= \BIN
	0 & A_{\bE_1}^T & \dots & 0 & 0 & \dots & 0 \\
	\vdots & \vdots & \ddots & \vdots & \vdots &  \ddots & \vdots  \\
	0 & 0 & \dots & A_{\bE_{p-1}}^T & 0 &  \dots & A_{\bE_{\cup p-2}}^T  
	\BOUT \nonumber\\
	B_{\phi} &= \BIN
	0 & I &0 &  \dots & 0 &\dots
	\BOUT\nonumber\\
	B_{\nu} &= \BIN
	\vdots & \vdots & \ddots & \vdots & \vdots & \vdots & \ddots  & \ddots & \vdots \\
	0 & 0 & \dots & 0 & 0 & \dots & I & \dots & 0\\
	\vdots & \vdots & \ddots & \vdots & \vdots  & \ddots & \vdots & \ddots & \vdots 
	\BOUT \nonumber\\
	C&\coloneqq\BIN C_{\mu}^T & C_{\eta}^T & C_{\phi}^T  & C_{\nu}^T \BOUT^T, \quad C_{\mu} = 0, \quad C_{\eta} = 0, \quad C_{\phi} =\BIN
	-I & \dots & 0 &  \dots
	\BOUT\nonumber\\
	C_{\nu} &=\BIN
	\vdots & \ddots & \vdots & \ddots\\
	0 & \dots & -I & \dots   \\
	\vdots & \ddots & \vdots & \ddots \\
	\BOUT \nonumber\\
	c &\coloneqq \BIN c_{\mu}^T & c_{\eta}^T & c_{\phi}^T & c_{\nu}^T  \BOUT^T, \quad
	c_{\mu} =  b_{\bE_{\cup p}}, \quad
	c_{\eta} = 0,\quad
	c_{\phi} =   -\hat{b}_{\bE_{\cup p-1}} ,\quad
	c_{\nu} = 0\nonumber
	\end{align}
	$I_{\epsilon}$ is an indicator function which projects the split variables onto their respective constraint sets:
	\begin{align}
	& I_{\epsilon}(z_{\bE_l},\tilde{\lambda}_{l,\bE_{\cup l-1}}) = 
	\begin{cases} 0 & \text{if}\quad   \hspace{3pt} z_{\bE_l}^Tz_{\bE_l} - \hat{b}_{\bE_l}^T\hat{b}_{\bE_l} + \tilde{\lambda}_{l,\bE_{\cup l-1}}^Tb_{\bE_{\cup l-1}} \leq 0  \nonumber	\\
	+\infty & \text{otherwise}
	\end{cases}
	\end{align}
	In the ADMM, an augmented Lagrangian is used, which writes as 
	\begin{align}
	\mathcal{L}_p^{\rho}(q,s,\tilde{x},u) &= \frac{1}{2}q^THq +  \sum_{l=1}^{p-1}I_{\epsilon}(z_{\bE_{l}},\tilde{\lambda}_{l,\bE_{\cup l-1}}) + \sigma \Vert x - \tilde{x} \Vert_2^2+ \sum_{\bullet,*\in\mu,\eta,\phi,\nu} \frac{\rho\rho_{\bullet}}{2}\Vert B_{\bullet}q + C_{\bullet}s - c_{\bullet} + \frac{1}{\rho\rho_{\bullet}}{*}\Vert_2^2 
	\end{align}
	The symbol $\bullet$ is a place-holder for the different constraint sets $\mu$, $\eta$, $\phi$, and $\nu$, while $*$ is a place-holder for the Lagrange multipliers $\mu$, $\eta$, $\phi$, and $\nu$ associated with the same constraint sets. 
	The parameter $\sigma>0$ is a regularization parameter and ensures that the resulting KKT system is positive-definite.
	The parameter $\rho$ is the variable step-size parameter, which is updated as described in Sec.~\ref{sec:updaterho}. Additionally, for each constraint set, a constant pre-factor is introduced in order to create a bias towards activity of equality constraints ($\mu_{\bE}$, $\eta$) by choosing higher values of the step-size parameter, and a bias towards inactivity of inequality constraints ($\phi$, $\nu$), which may be inactive. In this work, the choice is
	\begin{align}
	\rho_{\mu_{\bE}} = 100, \qquad
	\rho_{\eta} = 10, \qquad
	\rho_{\phi} = \rho_{\nu} = 1\nonumber
	\end{align}
	$\rho_{\phi}$ and $\rho_{\nu}$ both represent the same quadratic dual inequality constraint, and are therefore chosen to the same smaller value.
	
	ADMM alternately updates estimates for the primal variables $q$ (see Sec. \ref{sec:primupdate}), split variables $s$ (see Sec. \ref{sec:updatesplit}), and dual variables $\mu$, $\eta$, $\phi$, and $\nu$ (see Sec. \ref{sec:dualascent}).
	Convergence is measured by the norm of the primal and dual residuals $k_{prim}$ and $k_{dual}$~\eqref{eq:kprimdual}, which needs to be smaller than a small threshold $\chi \ll 1$. The residuals represent the first-order optimality conditions of the unaugmented Lagrangian  of \eqref{eq:admmdhlsp} with respect to the dual and primal variables, respectively~\cite{osqp}:
	\begin{align}
	k_{prim} &\coloneqq 
	\BIN
	A_{\bE_{\cup p}}x - b_{\bE_{\cup p}} - v_{\bE_{\cup p}} + w_{\bE_{\cup p}}\\
	A_{\bE_{1}}^Tv_{\bE_1} \\
	\vdots\\
	A_{\bE_{p-1}}^Tv_{\bE_{p-1}} + A_{\bE_{\cup p-2}}^T\lambda_{p-1,\cup p-2}\\
	v_{\bE_{\cup p-1}} + \hat{b}_{\bE_{\cup p-1}} - z_{\bE_{\cup p-1}}\\
	\vdots\\
	\lambda_{l,\bE_{\cup l-1}} - \tilde{\lambda}_{l,\bE_{\cup l-1}}\\
	\vdots
	\BOUT\qquad l = 2,\dots,p-1
	\nonumber\\
	k_{dual} &\coloneqq 
	\BIN
	A_{\bE_{\cup p}}^T\mu_{p,\bE_{\cup p}}\\
	\vdots\\
	-\mu_{\bE_{l}} + A_{\bE_l}\eta_{l} + \phi_{\bE_{l}}\\
	\vdots\\
	v_{\bE_p}-\mu_{p,\bE_p}\\
	\vdots\\
	A_{\bE_{\cup l-1}}\eta_{l} + \nu_{l,\bE_{\cup l-1}}\\
	\vdots
	\BOUT \qquad l = 2,\dots,p-1
	\nonumber\\
	k_{prim,dual} &\coloneqq \BIN k_{prim} \\ k_{dual}\BOUT\label{eq:kprimdual}
	\end{align}

	\subsection{Primal update}
	\label{sec:primupdate}
	The primal update for $q$ results from the condition \begin{align}
	\nabla_{q^{(i+1)}} \mathcal{L}_p^{\rho}(q^{(i+1)},s^{(i)},\tilde{x}^{(i)},u^{(i)}) =  \left(H + \sigma\BIN I_x \\ & 0 \BOUT \right)q + \sum_{\bullet,*\in\mu,\eta,\phi,\nu} \frac{\rho\rho_{\bullet}}{2}B_{\bullet}^T\left( B_{\bullet}q + C_{\bullet}s - c_{\bullet} + \frac{1}{\rho\rho_{\bullet}}{*} \right) 
	\end{align}
	In the following, the index $i$ is omitted for better readability. The above condition results in the linear system
	\begin{align}
	K_q^{\rho}q = -k_q^{\rho}
	\end{align} 
	with
	\begin{align}
	K_q^{\rho}&\coloneqq\resizebox{.85\hsize}{!}{$ \left[\begin{array}{@{}cccccccccccccccccccccc@{}}
		\frac{\rho_{\mu}}{1+\rho\rho_{\mu}}A_{\bE_p}^TA_{\bE_p} + \sum_{l}^{p-1}\rho_{\mu} A_{\bE_l}^TA_{\bE_l}  + \sigma I & -\rho_{\mu} A_{\bE_1}^T & -\rho_{\mu} A_{\bE_2}^T & \dots & -\rho_{\mu} A_{\bE_p}^T & 0 & \dots & 0 & 0 & 0\\
		-\rho_{\mu} A_{\bE_1} & \rho_{\phi}I + \rho_{\mu}I  + \rho_{\eta} A_{\bE_1}A_{\bE_1}^T & 0 & \dots & 0 & 0 & \dots & 0 & 0 & 0 \\
		-\rho_{\mu} A_{\bE_2} & 0 & \rho_{\phi}I + \rho_{\mu}I + \rho_{\eta} A_{\bE_2}A_{\bE_2}^T  & \dots & 0 & \rho_{\eta}A_{\bE_2}A_{\bE_1}^T& \dots & 0 & 0 & 0 \\
		\vdots & \vdots& \vdots & \ddots & \vdots & \vdots & \ddots & \vdots & \ddots & \vdots\\
		-\rho_{\mu} A_{\bE_{p-1}} & 0 & 0 & \dots & \rho_{\phi}I + \rho_{\mu}I + \rho_{\eta} A_{\bE_{p-1}}A_{\bE_{p-1}}^T  & 0 & \dots  & \rho_{\eta}A_{\bE_{p-1}}A_{\bE_1}^T & \dots & \rho_{\eta}A_{\bE_{p-1}}A_{\bE_p-2}^T\\
		0 & 0 & \rho_{\eta}A_{\bE_1}A_{\bE_2}^T & \dots & 0 & \rho_{\eta}A_{\bE_1}A_{\bE_1}^T + \rho_{\nu}I& \dots & 0 & \dots & 0 \\
		\vdots & \vdots& \vdots & \ddots & \vdots & \vdots & \ddots & \vdots & \ddots & \vdots\\
		0 & 0 & 0 & \dots & \rho_{\eta}A_{\bE_1}A_{\bE_{p-1}}^T & 0 & \dots  & \rho_{\eta}A_{\bE_1}A_{\bE_1}^T + \rho_{\nu}I& \dots & \rho_{\eta}A_{\bE_1}A_{\bE_{p-2}}^T\\
		\vdots & \vdots& \vdots & \ddots & \vdots & \vdots & \ddots & \vdots & \ddots & \vdots\\
		0 & 0 & 0 & \dots & \rho_{\eta}A_{\bE_{p-2}}A_{\bE_{p-1}}^T & 0 & \dots &  \rho_{\eta}A_{\bE_{p-2}}A_{\bE_1}^T & \dots & \rho_{\eta}A_{\bE_{p-2}}A_{\bE_{p-2}}^T + \rho_{\nu}I
		\end{array}\right]$}
	\label{eq:Kqrho}\\
	k_q^{\rho}&\coloneqq-\BIN 
	\sum_{l=1}^{p-1}\rho_{\mu} A_{\bE_l}^T(b_{\bE_l}  - \frac{1}{\rho_{\mu}}\mu_{\bE_l})
	+ \frac{\rho_{\mu}}{1+\rho\rho_{\mu}} A_{\bE_p}^T( b_{\bE_p}  - \frac{1}{\rho_{\mu}}\mu_{\bE_p}) + \sigma \tilde x\\ 
	\rho_{\mu}(-b_{\bE_1} + \frac{1}{\rho_{\mu}}\mu_{\bE_1}) - A_{\bE_1}\eta_{1} + \rho_{\phi}(z_{\bE_1}-\hat{b}_{\bE_1} - \frac{1}{\rho_{\phi}}{\phi}_{1})\\
	\rho_{\mu}(-b_{\bE_2} + \frac{1}{\rho_{\mu}}\mu_{\bE_2}) - A_{\bE_2}\eta_{2} + \rho_{\phi}(z_{\bE_2}-\hat{b}_{\bE_2} - \frac{1}{\rho_{\phi}}{\phi}_{2})\\
	\vdots\\
	\rho_{\mu}( -b_{\bE_{p-1}} + \frac{1}{\rho_{\mu}}\mu_{\bE_{p-1}}) - A_{\bE_{p-1}}\eta_{{p-1}} + \rho_{\phi}(z_{\bE_{p-1}}-\hat{b}_{\bE_{p-1}} - \frac{1}{\rho_{\phi}}{\phi}_{{p-1}})\\
	-A_{\bE_1}\eta_{2} +\rho_{\nu}( \tilde{\lambda}_{2,\bE_1} - \frac{1}{\rho_{\nu}}\nu_{2,\bE_1})\\
	-A_{\bE_1}\eta_{3} +\rho_{\nu}( \tilde{\lambda}_{3,\bE_1} - \frac{1}{\rho_{\nu}}\nu_{3,\bE_1})\\
	-A_{\bE_2}\eta_{3} +\rho_{\nu}( \tilde{\lambda}_{3,\bE_2} - \frac{1}{\rho_{\nu}}\nu_{3,\bE_2})\\
	\vdots \\
	-A_{\bE_1}\eta_{p-1} + \rho_{\nu}(\tilde{\lambda}_{p-1,\bE_1} - \frac{1}{\rho_{\nu}}\nu_{p-1,\bE_1})\\
	\vdots \\
	-A_{\bE_{p-2}}\eta_{p-1} + \rho_{\nu}(\tilde{\lambda}_{p-1,\bE_{p-2}} - \frac{1}{\rho_{\nu}}\nu_{p-1,\bE_{p-2}})
	\BOUT
	\end{align}
	The above uses the substitution
	\begin{align}
	v_{\bE_p} =& \frac{\rho\rho_{\mu}}{I+\rho\rho_{\mu} I}(A_{\bE_p}x-b_{\bE_p} + \frac{1}{\rho_{\mu}}\mu_{\bE_p})
	\label{eq:subvp}
	\end{align}
	The matrix $K_q^{\rho}$ is positive definite due to the introduction of the split variables $\tilde{x}$ and $\tilde{\lambda}_{l,\bC_{\cup l-1}}$. This renders the block-diagonals invertible for any rank of $A_{\cup p}$ for $\sigma > 0$ and $\rho_{\nu}>0$.
	In the following, it is shown how this can be exploited to reduce this KKT system to only depend on the primal variables $x$, while substitutions in dependence of $x$ are found for all other primal variables in $q$ (including the primal-dual variables $\Lambda_{p-1}$). This facilitates the factorization of the matrix
	
	The following abbreviations are introduced for later use:
	\begin{align}
	h_l &=  b_{\bE_l}  - \frac{1}{\rho_{\mu}}\mu_{\bE_l},\qquad l = 1,\dots,p\nonumber\\
	h_{l,primal} &= 
	\rho_{\mu}\left(-b_{\bE_l} + \frac{1}{\rho_{\mu}}\lambda_{\bE_l,1}\right) - A_{\bE_l}\eta_{l} + \rho_{\phi}\left(z_{\bE_l}-\hat{b}_{\bE_l} - \frac{1}{\rho_{\phi}}{\phi}_{l}\right),\qquad l = 1,\dots,p-1\nonumber\\
	h_{l,dual} &= 
	\BIN
	-A_{\bE_1}\eta_{l} + \rho_{\nu}\left(\tilde{\lambda}_{l,\bE_1} - \frac{1}{\rho_{\nu}}\nu_{l,\bE_1}\right)\\
	\vdots \\
	-A_{\bE_{l-1}}\eta_{l} + \rho_{\nu}\left(\tilde{\lambda}_{l,\bE_{l-1}} - \frac{1}{\rho_{\nu}}\nu_{l,\bE_{l-1}}\right)
	\BOUT,\qquad l = 2,\dots,p-1\nonumber
	\end{align}
	
	\subsubsection{Substitution of primal-dual variables}
	\label{sec:subpd}
	Substitutions for the primal-dual variables $\lambda_{l,\bE_{\cup l-1}}$ with $l=1,\dots,p-1$ can be found by analyzing the rows and columns of $K_q$ corresponding to the primal-dual. Here, the blocks on the corresponding diagonals $M$ write as
	\begin{align}
	M_l =
	\BIN M_{l-1} & T_{l}^T \\ T_l & U_l \BOUT \coloneqq
	\BIN  \rho_{\eta}A_{\bE_1}A_{\bE_1}^T + \rho_{\epsilon}I& \rho_{\eta}A_{\bE_1}A_{\bE_2}^T& \hdots & \rho_{\eta}A_{\bE_1}A_{\bE_l}^T\\
	\rho_{\eta}A_{\bE_2}A_{\bE_1}^T & \rho_{\eta}A_{\bE_2}A_{\bE_2}^T + \rho_{\epsilon}I& \hdots &  \rho_{\eta}A_{\bE_2}A_{\bE_l}^T\\
	\vdots & \vdots & \ddots & \vdots\\
	\rho_{\eta}A_{\bE_l}A_{\bE_1}^T & \rho_{\eta}A_{\bE_l}A_{\bE_2}^T & \hdots &  \rho_{\eta}A_{\bE_l}A_{\bE_l}^T + \rho_{\epsilon}I
	\BOUT \nonumber
	\end{align}
	with 
	\begin{align}
	M_{l-1}&\coloneqq	\BIN  \rho_{\eta}A_{\bE_1}A_{\bE_1}^T + \rho_{\epsilon}I& \rho_{\eta}A_{\bE_1}A_{\bE_2}^T& \hdots & \rho_{\eta}A_{\bE_1}A_{\bE_l}^T\\
	\rho_{\eta}A_{\bE_2}A_{\bE_1}^T & \rho_{\eta}A_{\bE_2}A_{\bE_2}^T + \rho_{\epsilon}I& \hdots &  \rho_{\eta}A_{\bE_2}A_{\bE_{l-1}}^T\\
	\vdots & \vdots & \ddots & \vdots\\
	\rho_{\eta}A_{\bE_{l-1}}A_{\bE_1}^T & \rho_{\eta}A_{\bE_{l-1}}A_{\bE_2}^T & \hdots &  \rho_{\eta}A_{\bE_{l-1}}A_{\bE_{l-1}}^T + \rho_{\epsilon}I
	\BOUT \nonumber\\
	T_l &= \BIN \rho_{\eta}A_{\bE_l}A_{\bE_1}^T & \rho_{\eta}A_{\bE_l}A_{\bE_2}^T & \hdots & \rho_{\eta}A_{\bE_l}A_{\bE_{l-1}}^T  \BOUT\nonumber\\
	U_l &= \rho_{\eta}A_{\bE_l}A_{\bE_l}^T + \rho_{\epsilon}I\nonumber
	\end{align}
	Using the relationship for block-inverses, the inverse of $M_l$ writes as
	\begin{align}
	M_l^{-1} = 
	\BIN
	M_{l-1}^{-1} + M_{l-1}^{-1}T_l^T(U_l-T_lM_{l-1}^{-1}T_l^T)^{-1}T_lM_{l-1}^{-1} & -M_{l-1}^{-1}T_l^T(U_l-T_lM_{l-1}^{-1}T_l^T)^{-1} \\
	-(U_l-T_lM_{l-1}^{-1}T_l^T)^{-1}T_lM_{l-1}^{-1} & (U_l-T_lM_{l-1}^{-1}T_l^T)^{-1}
	\BOUT \nonumber
	\end{align}
	It can be observed that the block diagonal inversion of a block $M_l$ can be done by reusing the previously calculated inversion of $M_{l-1}$. Reusing block inversions results in a overall number of operations of $\sum_{l=1}^{p-1} m_l^3$, instead of $\sum_{l=2}^{p} n_{l}^{dual,3}$ in case that all diagonal blocks are inverted from scratch. While this is still expensive, this computation only needs to be done once when solving a HLSP. As can be seen, the matrices $M$ are independent from $\rho$ and therefore do not change after step-size parameter updates.
	
	The primal-dual variables can be recovered with the following relationship
	\begin{align}
	\lambda_{l,\bE_{\cup l-1}}
	=
	M_l^{-1}\left(h_{l,dual} - 
	\rho_{\eta} A_{\bE_{\cup l-1}}A_{\bE_l}^T
	v_{\bE_l}\right)\qquad \quad
	l = 2,\dots,p-1
	\label{eq:comppd}
	\end{align}

	\subsubsection{Eliminating primal slacks of lower priority levels}
	\label{sec:subvl}
	The primal slacks $v_{\bE_{\cup p-1}}$ of lower priority levels can be computed from the primal $x$ by considering the rows and column of $K_q^{\rho}$ corresponding to the primal slacks $v_{\bE_{\cup p-1}}$. This leads to the following substitution in dependence of $x$:
	\begin{align}
	v_{\bE_l} = &\left(\rho I + \rho_{\epsilon} I + \rho_{\eta} \left(A_{\bE_l}A_{\bE_l}^T -
	A_{\bE_l}A_{\bE_{\cup l-1}}^T
	M_l^{-1}
	A_{\bE_{\cup l-1}}A_{\bE_l}^T\right)\right)^{-1}\left(\rho_{\mu} A_{\bE_l}x + h_{l,primal}  - \rho_{\eta}
	A_{\bE_l}A_{\bE_{\cup l-1}}^T 
	M_l^{-1}h_{l,dual}\right)\nonumber
	\end{align}

	\subsubsection{Reduced system}
	The overall reduced KKT system solely depends on the primal variable $x$, and writes as
	\begin{align}
	K_x^{\rho} x^{(i+1)} = -k_x^{\rho}
	\label{eq:Kx}
	\end{align}
	with
	\begin{align}
	K_x^{\rho} &= \rho_{\mu}A_{\bE_1}^T\left(I-(\rho_{\epsilon}I + \rho_{\mu}I +  \rho_{\eta}A_{\bE_1}A_{\bE_1}^T)^{-1}\rho_{\mu}\right)A_{\bE_1} \label{eq:Kxrho}\\
	&+ \rho_{\mu}A_{\bE_2}^T\left(I-(\rho_{\epsilon}I + \rho_{\mu}I + \rho_{\eta} A_{\bE_2}A_{\bE_2}^T - \rho_{\eta}^2 A_{\bE_2}A_{\bE_1}^T(A_{\bE_1}A_{\bE_1}^T + I)^{-1} A_{\bE_1}A_{\bE_2}^T )^{-1}\rho_{\mu}\right)A_{\bE_2} \nonumber\\
	& + \dots\nonumber\\
	& + \rho_{\mu}A_{\bE_{p-1}}^T\left(I - \left(\rho_{\epsilon}I + \rho_{\mu}I +  \rho_{\eta}A_{\bE_{p-1}}A_{\bE_{p-1}}^T - 
	\rho_{\eta}^2 A_{\bE_{p-1}}A_{\bE_{\cup{p-2}}}^T
	M_{p-2}^{-1}
	A_{\bE_{\cup{p-2}}}A_{\bE_{p-1}}^T\right)^{-1}\rho_{\mu}\right)A_{\bE_{p-1}} \nonumber\\
	&+ \frac{\rho_{\mu}}{1+\rho\rho_{\mu}}A_{\bE_{p}}^TA_{\bE_{p}}\nonumber
	\end{align}
	and
	\begin{align}
	k_x^{\rho} =& 
	\rho_{\mu}A_{\bE_1}^T(h_{1}  + (\rho_{\epsilon}I + \rho_{\mu}I + \rho_{\eta}A_{\bE_1}^TA_{\bE_1})^{-1}h_{1,primal})\label{eq:kxrho}\\
	+& \rho_{\mu}A_{\bE_2}^T(h_{2} + (\rho_{\epsilon}I + \rho_{\mu}I + \rho_{\eta} A_{\bE_2}A_{\bE_2}^T - \rho_{\eta}^2 A_{\bE_2}A_{\bE_1}^T(A_{\bE_1}A_{\bE_1}^T + I)^{-1} A_{\bE_1}A_{\bE_2}^T )^{-1}\nonumber\\
	&(h_{2,primal} - \rho_{\eta} A_{\bE_2}A_{\bE_1}^T(A_{\bE_1}A_{\bE_1}^T + I)^{-1}h_{2,dual})\nonumber\\
	+&\dots\nonumber\\
	+& \rho_{\mu}A_{\bE_{p-1}}^T\left(h_{p-1} + \left(\rho_{\epsilon}I + \rho_{\mu}I + \rho_{\eta} A_{\bE_{p-1}}A_{\bE_{p-1}}^T -
	\rho_{\eta}^2 A_{\bE_{p-1}}A_{\bE_{\cup p-2}}^T
	M_l^{-1}
	A_{\bE_{\cup p-2}}A_{\bE_{p-1}}^T\right)^{-1}\right.\nonumber\\
	&\left.\left(h_{p-1,primal}  - \rho_{\eta} A_{\bE_{p-1}}A_{\bE_{\cup p-2}}^T
	M_l^{-1}h_{l,dual}\right)\right)\nonumber\\
	+&\frac{\rho_{\mu}}{1+\rho\rho_{\mu}}A_{\bE_p}^Th_{p}\nonumber
	\end{align}
	$K_x^{\rho}$ is positive-definite for $\sigma > 0$. 	The linear system \eqref{eq:Kx} can be solved using a Cholesky or LDLT decomposition, where the factorization of $K_x^{\rho}$ only needs to be updated if $\rho$ changes.
	
	\subsection{Updating primal splitting variables}
	\label{sec:updatesplit}
	For the splitting variable update, we use the condition $\nabla_{s^{(i+1)},\tilde{x}^{(i+1)}} \mathcal{L}_{\rho}(q^{(i+1)},s^{(i+1)},\tilde{x}^{(i+1)},*^{(i)}) = 0$ (with $*\in \{\mu,\eta,\phi,\nu\}$). For one, this results in the update $\tilde{x}^{(i+1)} = x^{(i+1)}$. Furthermore, considering the fact that $C_{\bullet}^TC_{\bullet}=I$, this condition becomes $\rho_{\bullet} C_{\bullet}^TC_{\bullet}s = \rho_{\bullet} C_{\bullet}^T(-B_{\bullet}q + c_{\bullet} - \frac{1}{\rho\rho_{\bullet}}*)$ (with $*,\bullet\in \{\mu,\eta,\phi,\nu\}$), which fully writes as
	\begin{align}
	z_{\bE_l} &= v_{\bE_l} + \hat{b}_{\bE_l} + \frac{1}{\rho\rho_{\phi}}\phi_{\bE_l,\cup l-1}, \qquad l = 2,\dots,p-1\label{eq:splitupdate}\\
	\tilde{\lambda}_{l,\bE_{\cup l-1}} &= \lambda_{l,\bE_{\cup l-1}} + \frac{1}{\rho\rho_{\nu}}\nu_{l,\bE_{\cup l-1}}, \qquad l = 2,\dots,p-1\nonumber
	\end{align}
	$z_{\bE_l}$ with $l=1,\dots,p-1$ and $\tilde{\lambda}_{l,\bE_{\cup l-1}}$ are subject to the convex quadratic projection $I_{\epsilon}(z_{\bE_l},\tilde{\lambda}_{l,\bE_{\cup l-1}})$.
	This can be achieved by solving the following convex QCLSP for $z_{\bE_l}$ ($l=1,\dots,p-1$) and $\tilde{\lambda}_{l,\bE_{\cup l-1}}$ ($2=1,\dots,p-1$)
	\begin{align}
	\mini_{z_{\bE_l},\lambda_{l,\bE_{\cup l-1}}} \qquad& \frac{1}{2}\left\Vert 
	\BIN z_{\bE_l} \\ \tilde{\lambda}_{l,\bE_{\cup l-1}} \BOUT
	-
	\BIN
	a_{1,l}\\
	a_{2,\cup l-1}
	\BOUT
	\right\Vert_2^2
	\label{eq:qcproj}\tag{QCLSP$-{\epsilon}$}\\
	s.t.\qquad & z_{\bE_l}^Tz_{\bE_l} - \hat{b}_{\bE_l}^T\hat{b}_{\bE_l} + \tilde{\lambda}_{l,\bE_{\cup l-1}}^Tb_{\bE_{\cup l-1}} \leq 0\nonumber
	\end{align}
	Following abbreviations are used:
	\begin{align}
	a_{1,l} \coloneqq& v_{\bE_l} + \hat{b}_{\bE_l} + \frac{1}{\rho_{\phi}}\phi_{\bE_l}\nonumber\\
	a_{2,\cup l-1} \coloneqq& \lambda_{l,\bE_{\cup l-1}} + \frac{1}{\rho_{\nu}}\nu_{l,\bE_{\cup l-1}}\nonumber
	\end{align}
	The Lagrangian of \ref{eq:qcproj} then writes as
	\begin{align}
	\mathcal{L}_l^{\epsilon} =& \frac{1}{2}\left\Vert 
	\BIN z_{\bE_l} \\ \tilde{\lambda}_{l,\bE_{\cup l-1}} \BOUT
	-
	\BIN
	a_{1,l}\\a_{2,\cup l-1}
	\BOUT
	\right\Vert_2^2 
	+
	\theta_{p,l} (z_{\bE_l}^Tz_{\bE_l} - \hat{b}_{\bE_l}^T\hat{b}_{\bE_l} + \tilde{\lambda}_{l,\bE_{\cup l-1}}^Tb_{\bE_{\cup l-1}})	
	\end{align}
	The Lagrange multiplier $\theta_{p,l}$ is subject to the condition $\theta_{p,l} \geq 0$ in order to ensure the negativity of the quadratic inequality constraints.
	From the first-order derivative of $\mathcal{L}_l^{\epsilon}$ with respect to $z_{\bE_l}$, the following expression is obtained:
	\begin{align}
	z_{\bE_l} =& \frac{a_{1,l}}{1+2\theta_{p,l}}\nonumber
	\end{align}
	The first order derivative with respect to $\tilde{\lambda}_{{\cup l-1},\bE_{\cup l-1}}$ yields
	\begin{align}
	\tilde{\lambda}_{l,\bE_{\cup l-1}} =& a_{2,\cup l-1} - \theta_{p,l} b_{\bE_{\cup l-1}}\nonumber
	\end{align}
	Finally, the first-order derivative with respect to $\theta_{p,l}$ using the expressions for $z_{\bE_l}$ and $\lambda_{{\cup l-1},\bE_{\cup l-1}}$ gives
	\begin{align}
	e_{1,l} + \theta_{p,l} e_{2,l} + \theta_{p,l}^2 e_{3,l} + \theta_{p,l}^3 e_{4,l} = 0 \label{eq:qcqppoly}
	\end{align}
	with
	\begin{align}
	e_{1,l} = d_{3,l}- d_{2,l}, \qquad 
	e_{2,l} = -d_{1,l} - 4d_{2,l}, \qquad
	e_{3,l} = -4d_{1,l} - 4d_{2,l}, \qquad
	e_{4,l} = -4d_{1,l}\nonumber
	\end{align}
	and
	\begin{align}
	d_{1,l} = b_{\bE_{\cup l-1}}^Tb_{\bE_{\cup l-1}},\qquad
	d_{2,l} = \hat{b}_{\bE_l}^T\hat{b}_{\bE_l} - a_{2,\cup l-1}^Tb_{\bE_{\cup l-1}},\qquad
	d_{3,l} = a_{1,l}^Ta_{1,l}\nonumber
	\end{align}
	If $d_{1,l} \neq 0$, the above is a third-order polynomial \eqref{eq:qcqppoly}  and can be solved analytically, for example by Cardano's formula. The solution consists of at least one real solution for $\theta_{p,l}$, and at most three different real solutions. A positive $\theta_{p,l}\geq 0$ indicates zero duality gap $\frac{1}{2}\Vert v_{\bE_l}\Vert^2 = g_l({v}_{\mathbb{E}_{\cup l}},\lambda_{l,\bE_{\cup l-1}})$. The condition $\theta_{p,l} \geq 0$ is enforced by $\theta_{p,l} \leftarrow \max(0,\theta_{p,l})$.

	\subsection{Partial preconditioning for computational efficiency}
	\label{sec:precond}
	The number of iterations $i$ in first-order methods depends on the conditioning of the problem at hand. A popular since computationally reasonable approach for improving the conditioning of optimization problems  is Ruiz matrix equilibration~\cite{ruizscaling}. Here, matrix rows and columns are iteratively scaled until they are normalized. The authors in~\cite{osqp} suggest to scale the  KKT system of \eqref{eq:admmdhlsp}, which is scaled as follows:
	\begin{align}
	\BIN \overline{H} & \overline{B}^T \\ \overline{B} & 0\BOUT =	\BIN L & 0 \\ 0 & V \BOUT \BIN P & B^T\\ B & 0 \BOUT \BIN L & 0 \\ 0 & V \BOUT
	\label{eq:equPA}
	\end{align}
	with 
	\begin{align}
	& \overline{H} = LHL\nonumber\\
	& \overline{B} = VBL\nonumber
	\end{align}
	Primal variables $q$, dual variables $\bullet \in \{\mu,\eta,\phi,\nu\}$, vectors $c$, and the convex set $\mathcal{C}$ are scaled as follows
	\begin{align}
	&\overline{q} = L^{-1}q\\
	&\overline{\bullet} = V^{-1}\bullet\\
	&\overline{c} = Vc\\
	&\overline{\mathcal{C}} = \{Vs | s\in\mathcal{C}\}
	\label{eq:equVar}
	\end{align}
	Overlined matrices and variables are the equilibrated equivalents of the original quantities.
	$L$ and $V$ are subdivided into components corresponding to their respective variable or constraint sets $\{x,v,\nu\}$ and $\{\mu,\eta,\phi,\nu\}$, respectively:
	\begin{align}
	L &\coloneqq \BIN L_x^T\in\mathbb{R}^{n_x\times n_x} & L_v^T\in\mathbb{R}^{m_{\cup p }\times m_{\cup p }} & L_{\nu}^T \in \mathbb{R}^{\sum_{l=2}^{p-1} n_l^{dual}\times\sum_{l=2}^{p-1} n_l^{dual}} \BOUT^T\nonumber\\
	V & \coloneqq \BIN V_{\mu}^T\in\mathbb{R}^{m_{\cup p}\times m_{\cup p}} & V_{\eta}^T\in\mathbb{R}^{(p-1)n_x \times (p-1)n_x}& V_{\phi}^T \in\mathbb{R}^{m_{\cup p-1} \times m_{\cup p-1}} & V_{\nu}^T \in \mathbb{R}^{\sum_{l=2}^{p-1} n_l^{dual}\times \sum_{l=2}^{p-1} n_l^{dual}} \BOUT^T\nonumber
	\end{align}
	Ruiz equilibraton could be applied to the full KKT system as is depicted in~\eqref{eq:equPA}. However, this would significantly increase the computational burden of the proposed solver due to two issues. 
	
	For one, the diagonal blocks corresponding to the primal-dual variables possibly consist of differently equilibrated matrices $A_{\bE_{\cup p}}$. This prevents the efficient since recursive computation of the inverses of the dual diagonal blocks $M$ (see Sec. \ref{sec:subpd}). In order to avoid this, only the constraint matrices $A_{\bE_{\cup p}}$ are scaled by Ruiz equilibration
	\begin{align}
	\overline{A}_{\bE_{\cup p}}  =	V_{\mu}  {A}_{\bE_{\cup p}} L_x 
	\end{align}
	The same equilibration is applied to the constraint matrices corresponding to the dual constraints by choosing
	\begin{align}
	L_v &= I \nonumber\\
	L_{\nu} &= \BIN V_{\mu}(1:m_{\cup 1},1:m_{\cup 1})^T & \dots & V_{\mu}(1:m_{\cup p-2},1:m_{\cup p-2})^T \BOUT^T \nonumber\\
	V_{\eta} &= \BIN L_x^T & \dots & L_x \BOUT^T \nonumber\\
	V_{\phi} &= L_{v}^{-1} \nonumber\\
	V_{\nu,l} &= L_{\nu}\nonumber
	\end{align}
	$V_{\mu}(a:b,c:d)$ extracts a sub-matrix from row $a$ to $b$ and column $c$ to $d$ from the original matrix $V_{\mu}$.
	The choices for $L_{\nu}$ and $V_{\eta}$ ensure that the occurrence of the matrix $\overline{A}_{\bE_l}$ on the diagonal of level $l=1,\dots p-2$ in the KKT Hessian has the same scaling. Otherwise, each matrix $M_l$ with $l=2,\dots,p-1$ in Sec. \ref{sec:subpd} would need to be computed from scratch, instead of the proposed efficient recursive computation. The choice for $L_{v} = I$ (instead of $L_{v}=V_{\mu}$, which would yield the same scaling for the off-diagonal occurrences of $A_{\bC_{\cup p}}$) avoids the quadratic cross-term $V_{\mu}V_{\mu}$, which would amplify possibly bad conditioning in $A_{\bE_{\cup p}}$. 
	
	Secondly, scaling the entire KKT Hessian by Ruiz equilibration negatively influences the computational complexity of computing the quadratic projection \ref{eq:qcproj}, as is demonstrated in the following.
	\begin{theorem}
		If $V_{\phi}\neq I$ and $V_{\nu}\neq I$, a polynomial of order $2m_{l}+1$  needs to be solved for the quadratic projection of each level $l=1,\dots,p-1$.
		\label{th:precond}
	\end{theorem}
	\begin{proof}
		According to~\eqref{eq:equVar}, the scaled quadratic projection \ref{eq:qcproj} becomes 
		\begin{align}
		\mini_{z_l,\lambda_{l,\bE_{\cup l-1}}} \qquad& \frac{1}{2}\left\Vert 
		\BIN V_{\phi,l}z_{\bE_l} \\ V_{\nu,l}\tilde{\lambda}_{l,\bE_{\cup l-1}} \BOUT
		-
		\BIN
		\overline{a}_{1,l}\\
		\overline{a}_{2,\cup l-1}\\
		\BOUT
		\right\Vert_2^2\label{eq:qcqpequ}\\
		s.t.\qquad & z_{\bE_l}^Tz_{\bE_l} - \hat{b}_{\bE_l}^T\hat{b}_{\bE_l} + \tilde{\lambda}_{l,\bE_{\cup l-1}}^Tb_{\bE_{\cup l-1}} \leq 0
		\nonumber
		\end{align}
		with the abbrevations
		\begin{align}
		\overline{a}_{1,l}&\coloneqq F_l\overline{v}_{\bE_l} + V_{\phi,l}\hat{b}_{\bE_l} + \frac{1}{\rho_{\mu}}\overline{\lambda}_{p,\bE_l}\nonumber\\
		\overline{a}_{2,\cup l-1} &\coloneqq J_l\overline{\lambda}_{2,\bE_1} + \frac{1}{\rho_{\nu}}\overline{\nu}_{2,\bE_1} \nonumber
		\end{align}
		The additional index $l$ produces the components of $V_{\phi,l}$ and $V_{\nu,l}$ which correspond to the constraint sets $\phi$ and $\nu$ of level $l$. 
		The matrices $F_l$ and $J_l$ are scaling matrices $F_l = V_{\phi,l} L_{v,l}$ and $J_l = V_{\nu,l} L_{\nu,l}$, respectively, which are individual blocks contained in the scaled matrix $\overline{B}$. 
		Similarly to Sec. \ref{sec:updatesplit}, the split variables are computed as
		\begin{align}
		z_{\bE_l} =& (V_{\phi,l}^TV_{\phi,l}+2\theta_{p,l} I)^{-1}V_{\phi,l}\overline{a}_1\\
		\tilde{\lambda}_{l,\bE_{\cup l-1}} =& (V_{\nu,l}^TV_{\nu,l})^{-1}(V_{\nu,l}^T\overline{a}_2 - \theta_{p,l} b_{\bE_{\cup l-1}})
		\end{align}
		Inserting these expressions into the first-order optimality condition of the scaled QCLSP \eqref{eq:qcqpequ} with respect to $\theta_{p,l}$ gives
		\begin{align}
		\sum_{j=1}^{m_{l}} a_{1}(j)^2V_{\phi,l}(j)^2\frac{1}{(V_{\phi,l}(j)^2 + 2\theta_{p,l})^2} + \sum_{j=1}^{m_{\cup l-1}} \frac{b_{\bE_{\cup l-1}}(j)(V_{\nu,l}(j)\overline{a}_{2,\cup l-1}(j) - \theta_{p,l} b_{\bE_{\cup l-1}}(j))}{V_{\nu,l}(j)^2} = 0
		\label{eq:scaledlambdacomp}
		\end{align}
		It can be observed that a common denominator for the first term in \eqref{eq:scaledlambdacomp} is	$\Pi_{i=1}^{m_l} (V_{\phi,l}(i)^2 + 2\lambda)^2$.
		If \eqref{eq:scaledlambdacomp} is multiplied by this polynomial of order $2m_l$, it leads to a polynomial of overall order $2m_{l}+1$, as claimed.
	\end{proof}
	Instead, the scaling matrices corresponding to the split constraints can be set as
	\begin{align}
	V_{\phi,l} &= I \qquad  \text{and} \qquad
	V_{\nu,l} = I\nonumber
	\end{align}
	In this case, \eqref{eq:scaledlambdacomp} becomes a third-order polynomial as in Sec. \ref{sec:updatesplit} and can be solved efficiently.
	All occurrences of identity matrices in \eqref{eq:Kqrho} or \eqref{eq:Kxrho} related to the parameters $\rho_{\phi}$ and $\phi_{\nu}$ are now weighted diagonal matrices $F$ and $J$, whose conditioning represents the (possibly bad) conditioning of the original unscaled matrix $B$.
	Another possibility is to employ an interior-point method (see App.~\ref{app:ipmepsilon}), which is free of matrix factorizations and therefore presents a valuable compromise between numerical stability and computational expenditure.
	
	\subsection{Dual ascent for dual Lagrange multiplier update}
	\label{sec:dualascent}
	The Lagrange multipliers are updated by dual ascent of the dual problem of the QCLSP, which writes as 
	\begin{align}
	* =& * + \rho\rho_{\bullet}\nabla_{*}g({*}) = {*} + \rho\rho_{\bullet}(B_{\bullet}q + C_{\bullet}s - c_{\bullet})
	\end{align}
	The symbols $*,\bullet\in\{\mu,\eta,\phi,\nu\}$ are place-holders for the individual dual variables and their corresponding related constraint sets. The updates are as follows:
	\begin{align}
	\mu_{\bE_{\cup p}} \pluseq& \rho\rho_{\mu}(A_{\bE_{\cup p}}x - b_{\bE_{\cup p}} - v_{\bE_{\cup p}})\nonumber\\
	\eta_{l} \pluseq& \rho\rho_{\eta}\left( A_{\bE_l}^Tv_{\bE_l} + A_{\cup l-1}^T\lambda_{l,\bE_{\cup l-1}}\right) \qquad l = 1,\dots,p-1\nonumber\\
	{\phi}_{\bE_{\cup p-1}} \pluseq& \rho\rho_{\phi}(v_{\bE_{\cup p-1}} + \hat{b}_{\bE_{\cup p-1}} - z_{\bE_{\cup p-1}})\nonumber\\
	\nu_{l,\bE_{\cup l-1}} \pluseq& \rho\rho_{\nu}(\lambda_{l,\bE_{\cup l-1}} - \tilde{\lambda}_{l,\bE_{\cup l-1}})\qquad l=2,\dots,p-1\nonumber
	\end{align}

	\subsection{Updating $\rho$}
	\label{sec:updaterho}
	The step-size parameter $\rho$ is updated according to the method presented in~\cite{osqp}
	\begin{align}
	\rho \leftarrow \rho 
	\sqrt{
		\frac{
			\Vert k_{prim}\Vert_{\infty} / \max{\Vert Bq-c\Vert_{\infty},\Vert Cs \Vert_{\infty}}
		}{
			\Vert k_{dual}\Vert_{\infty} / \max{\Vert Hx\Vert_{\infty},\Vert B^T\lambda\Vert_{\infty}}
	}}\nonumber
	\end{align}
	It balances reduction in the scaled primal and dual residuals.
	If $\rho$ is updated, $K^{\rho}_x$ needs to be updated and re-factorized. In contrast, $M_l$ for the substitution of the primal-dual variables (see Sec. \ref{sec:subpd}) remains unchanged, as it is independent of the variable step-size $\rho$, but only dependent on the constant ones $\rho_{\mu}$, $\rho_{\eta}$, $\rho_{\phi}$, and $\rho_{\nu}$.

	\subsection{Over-relaxed step}
	\label{sec:overrelax}
	In over-relaxed ADMM, a relaxed version of ${Bq^{i+1}} \leftarrow \alpha Bq^{i+1} - (1-\alpha)(Cs^{i} - c)$ with $\alpha \in (0,2)$ is used in the split variable and dual variable updates (sections \ref{sec:updatesplit} and \ref{sec:dualascent}). This has been observed to enhance convergence speed~\cite{Ghadimi2015}.
	\begin{align}
	A_{\bE_{\cup p}}x^{i+1} - v_{\bE_{\cup p}}^{i+1}&\leftarrow \alpha (A_{\bE_{\cup p}}x^{i+1} - v_{\bE_{\cup p}}^{i+1}) + (1-\alpha) b_{\bE_{\cup p}}\nonumber\\
	A_{\bE_l}^Tv_{\bE_l}^{i+1} + A_{\bE_{\cup l-1}}^T\lambda_{l,\bE_{\cup l-1}}^{i+1} &\leftarrow \alpha(A_{\bE_l}^Tv_{\bE_l}^{i+1} + A_{\bE_{\cup l-1}}^T\lambda_{l,\bE_{\cup l-1}}^{i+1}) = 0 \qquad l=1,\dots,p-1 \nonumber\\
	{v}_{\bE_l}^{i+1}& \leftarrow \alpha{v}_{\bE_l}^{i+1} - (1-\alpha)(-z_{\bE_l}^{i} + \hat{b}_{\bE_l})  \qquad l=1,\dots,p-1 \nonumber\\
	\lambda_{l,\bE_{\cup l-1}}^{i+1}&\leftarrow \alpha\lambda_{l,\bE_{\cup l-1}}^{i+1} + (1-\alpha) \tilde{\lambda}_{l,\bE_{\cup l-1}}^{i} \qquad l=2,\dots,p-1\nonumber
	\end{align}
	The split variable update becomes
	\begin{align}
	z_{\bE_l}^{i+1} &= \alpha{v}_{\bE_l}^{i+1} - (1-\alpha)(-z_{\bE_l}^{i} + \hat{b}_{\bE_l}) + \hat{b}_{\bE_l} + \frac{1}{\rho\rho_{\phi}}\phi_{\bE_l,\cup l-1}^{i}, \qquad l = 2,\dots,p-1\nonumber\\
	\tilde{\lambda}_{l,\bE_{\cup l-1}}^{i+1} &= \alpha\lambda_{l,\bE_{\cup l-1}}^{i+1} + (1-\alpha) \tilde{\lambda}_{l,\bE_{\cup l-1}}^{i}  + \frac{1}{\rho\rho_{\nu}}\nu_{l,\bE_{\cup l-1}}^{i}, \qquad l = 2,\dots,p-1\nonumber
	\end{align}
	The dual variable update writes as
	\begin{align}
	\mu_{\bE_{l}} \pluseq& \rho\rho_{\mu}(\alpha(A_{\bE_{\cup p}}^{i+1}x - v_{\bE_{\cup p}}^{i+1}) - (1-\alpha) b_{\bE_{\cup p}})\nonumber\\
	\eta_{l} \pluseq& \rho\rho_{\eta}\alpha\left( A_{\bE_l}^Tv_{\bE_l}^{i+1} + A_{\cup l-1}^T\lambda_{l,\bE_{\cup l-1}}^{i+1}\right) \qquad l = 1,\dots,p-1\nonumber\\
	{\phi}_{\bE_{\cup p-1}}^{i+1} \pluseq& \rho\rho_{\phi}(\alpha{v}_{\bE_l}^{i+1} - (1-\alpha)(-z_{\bE_l}^{i} + \hat{b}_{\bE_l}) + \hat{b}_{\bE_{\cup p-1}} - z_{\bE_{\cup p-1}}^{i+1})\nonumber\\
	\nu_{l,\bE_{\cup l-1}}^{i+1} \pluseq& \rho\rho_{\nu}(\alpha\lambda_{l,\bE_{\cup l-1}}^{i+1} + (1-\alpha) \tilde{\lambda}_{l,\bE_{\cup l-1}}^{i}  - \tilde{\lambda}_{l,\bE_{\cup l-1}}^{i+1})\qquad l=2,\dots,p-1\nonumber
	\end{align}
	
	\section{Comparison of hierarchical programming solvers with linear constraints}
	\label{sec:hpcompare}
	\begin{table*}[htp!]
		\centering
		\resizebox{\columnwidth}{!}	
		{%
			\begin{tabular}{@{} cccccccccccc @{}}  
				\toprule
				& \textbf{Norm} & \textbf{Type} & \textbf{Ineq.} &  \textbf{$\mathcal{N}$-proj.} & $\dim(K)$ & $\Lambda_{p-1}$ \textbf{comp.} & \textbf{Dense / Block op.} & \textbf{Diff.} & \textbf{Warm-start} & \textbf{Application}\\
				\midrule
				\textbf{HQP} \cite{escande2014,dimitrov:2015} & 2 & P & ASM  & \checkmark & $ m_{\mathcal{A}_l}\times n_x^{r,l}$ & $1\times$ & \checkmark / limited~\cite{dimitrov:2015} & \xmark & \xmark & \begin{tabular}{@{}c@{}}well-posed problems /\\small AS changes \end{tabular} \\
				\midrule
				$\mathcal{N}$\textbf{IPM} \cite{pfeiffer2021}& 2 & P & IPM  & \checkmark & \begin{tabular}{@{}c@{}}$(m_{\bE_l}+m_{\mI_{\cup l-1}}) \times n_x^{r,l}$ \\ or $n_x^{r,l}\times n_x^{r,l}$\end{tabular} & $1\times$ & \xmark & \xmark & \xmark & \begin{tabular}{@{}c@{}} ill-posed problems /\\ large AS changes\end{tabular}\\
				\midrule
				$\mathcal{N}$\textbf{ADM} \cite{pfeiffer2024b} & 2 & P & ADMM  & \checkmark & $n_x^{r,l}\times n_x^{r,l}$ & $1\times$ & \xmark & \xmark & \checkmark & MPC / embedded systems\\
				\midrule
				$\mathcal{N}$\textbf{QP} \cite{pfeiffer2025} & 0,1 & P & IPM  & \checkmark & $n_x^{r,l}\times n_x^{r,l}$ & $1\times$ & \xmark & \xmark & \xmark & sparse problems\\
				\midrule
				\textbf{HLP}	\cite{Sathya2021} & 1 & D & \begin{tabular}{@{}c@{}} Simplex /\\ IPM\end{tabular} & \xmark & $(n_x + m_{\cup p} + \sum_{l=1}^{p-2}n_l^{dual})^2$ &$i\times$  & \xmark & \checkmark & \checkmark / \xmark& sparse problems\\
				\midrule
				\textbf{D-HADM (this work)} & 2 & D & \xmark  & \xmark & $n_x \times n_x$& $i\times$ & \checkmark & \checkmark & \checkmark & reasonable speed, less accurate\\
				\midrule
				\textbf{D-HIPM (this work)} & 2 & D & \xmark  & \xmark & $(n_x + \sum_{l=1}^{p-2}n_l^{dual})^2$ & $i\times$ & \xmark & \checkmark & \xmark & slow, accurate\\
				\bottomrule
		\end{tabular}}
		\caption{Comparison of different hierarchical programming solvers. Criteria include the norm of the linear constraint relaxation (`Norm'), underlying problem formulation (primal or dual HLSP, `Type'), ability of and method for handling of  inequalities (`Ineq. Cstr.'), whether they include variable and constraint elimination through nullspace projections for computational efficiency (`$\mathcal{N}$-proj.'), the dimension of the KKT system after substitutions (`$\dim(K)$'), frequency of primal-dual updates (`$\Lambda_{p-1}$ comp.'), whether the solver is implemented with dense linear algebra routines and supports memory efficient block operations (`Dense / Block op.'), whether the solver is differentiable (`Diff.'), whether the solver can be warm-started (`Warm-start'), and finally a remark on the preferred application for the solver (`Application').}
		\label{tab:hlspcomp}
	\end{table*}
	Table~\ref{tab:hlspcomp} compares characteristics and the computational scope of different solvers for hierarchical programming with linear constraints.
	These programs are structured similarly to \ref{eq:lexhlsp}, but can be formulated in any norm $\ell$ for the relaxation of the linear constraints. HLSP solvers are the most common ones ($\ell=2$), but some advances in sparse programming ($\ell=0,1$) for hierarchical decision making have been made \cite{Sathya2021,pfeiffer2025}.
	The solvers HQP, $\mathcal{N}$IPM, $\mathcal{N}$QP, and $\mathcal{N}$ADM solve the primal HLSP \eqref{eq:hlsp}. They are based on the ASM, IPM, or ADMM for handling inequality constraints. The ASM is preferred for well-posed problems with small changes of the active set between problem instances, for example as HLSP sub-problems in non-linear HLSP solvers \cite{pfeiffer2024}. On the other hand, the IPM is more robust to ill-posed problems and is agnostic to the overall number of changes in the active-set. This holds similarly for the ADMM, which has the additional advantage that matrix factorizations can be computed offline, for example for embedded applications with limited computation power \cite{dang2015}. For embedded solutions, the ability to warm-start the solver with a guess of the active-set presents a further advantage of the ADMM or ASM.  All primal solvers can be reduced to the primal variables $x$ by appropriate variable substitution. Additionally, \ref{eq:hlsp} solvers eliminate active constraints $\mathcal{A}_{\cup l-1}$ from the remaining problem $l$ to $p$ by projection into a basis of their nullspace $N_{l-1}\in\mathbb{R}^{n_x\times n_{x}^{r,l}}$, such that $A_{\mathcal{A}_{\cup l-1}}N_{l-1}=0$. The number of remaining variables on each level $l=1,\dots,p$ is $n_x^{r,l}$ with $n_x^{r,1} = n_x$. The rank of the matrix $A_{\mathcal{A}_{\cup l-1}}$ is $n_x - n_x^{r,l}$. This leads to progressively smaller problems due to reduction both in the number of variables and constraints.  Additionally, the primal-dual variables associated with the active constraints are eliminated from the problem by virtue of the nullspace projection, and only need to be computed once the primal solution of the HLSP is obtained, and only if required. This is in contrast to solvers based on dual HLSP, which need to compute updates of the primal-dual variables in every solver iteration $i$.
	
	HQP~\cite{escande2014} uses the least-squares form where the dimension of the KKT system to be solved in each solver iteration is $n_x^{r,l}\times m_{\mathcal{A}_l}$. $m_{\mathcal{A}_l}$ is the number of active constraints on level $l$, which includes the equality constraints $\bE_l$. Active constraints $\mathcal{A}_{\cup l-1}$ have been eliminated by nullspace projection. 
	Solving the least-squares form requires a rank-revealing QR decomposition, which can be reused to compute a basis of the nullspace of the active constraints.
	$\mathcal{N}$IPM~\cite{pfeiffer2021} is implemented both in the least-squares form and the normal form. The least-squares form in the IPM is of dimension $(m_{\bE_l}+m_{\mI_{\cup l-1}}) \times n_x^{r,l}$ and is therefore larger than the one from the ASM as it includes inactive inequality constraints both from the current level $l$ and the previous ones $\cup l-1$. The normal form is of dimension $n_{x}^{r,l}\times n_{x}^{r,l}$ and is advantageous if the number of constraint is high on lower priority levels. Factorization of the (semi-)positive definite Lagrangian Hessian can be done with a less computation-intense Cholesky or LDLT decomposition. However, a basis of nullspace needs to be computed from scratch and cannot reuse the previously obtained matrix factorization.
	
	The HQP solver designed in~\cite{dimitrov:2015} uses a specific choice of nullspace basis, which promotes efficient level-wise block operations during the nullspace computation and projection. On the other hand, dual HLSP solvers can treat the constraint matrices of all priority levels as one block $A_{\bE_{\cup p}}$. This is exploited in the proposed solver D-HADM, which is implemented using dense matrix operations. The dual solver HLP~\cite{Sathya2021} is based on sparse linear algebra routines as Gurobi \cite{gurobi} is used to solve the linear programs either via the Simplex method or the IPM. D-HIPM is implemented using sparse matrix formulations as well. Due to the simultaneous treatment of all priority levels, nullspace projections for variable or constraint elimination cannot be used in dual HLSP solvers. Furthermore, the primal-dual variables need to be computed in every solver iteration, which requires expensive matrix-vector multiplications of order $O((\sum_{l=1}^{p-2}n_l^{dual})^2)$ (proposed dual solver D-HADM). Primal-dual computations are significantly more expensive in the dual solvers HLP \cite{Sathya2021} (based on gurobi \cite{gurobi} using the simplex method or IPM for linear programs) and D-HIPM (this work, see Sec. \ref{app:ipm}) as no variable substitutions can be found and  large matrices need to be factorized instead. The proposed solvers D-HADM and D-HIPM are not able to handle inequality constraints due to the presence of non-convex complementary constraints.

	\section{Evaluation}
	\label{sec:eval}
	
	\begin{figure}[htp!]
		\includegraphics[width=0.75
		\columnwidth]{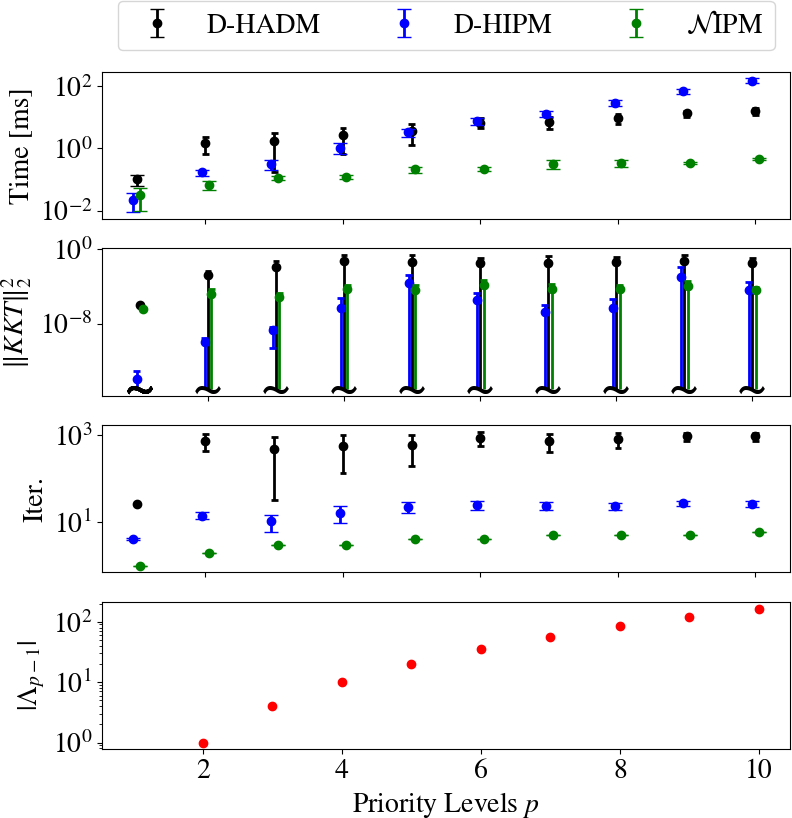}
		\centering
		\caption{Solver data of D-HADM, D-HIPM, and $\mathcal{N}$IPM~\cite{pfeiffer2021} for hierarchies of $p=1,\dots,10$ levels with equality constraints. Shown from top to bottom are the solver times, norm of the KKT residuals, number of solver iterations, and the number of primal-dual variables $\vert \Lambda_{p-1}\vert$.}
		\label{fig:solverdata}
	\end{figure}
	
	\begin{figure}[htp!]
		\includegraphics[width=0.475
		\columnwidth]{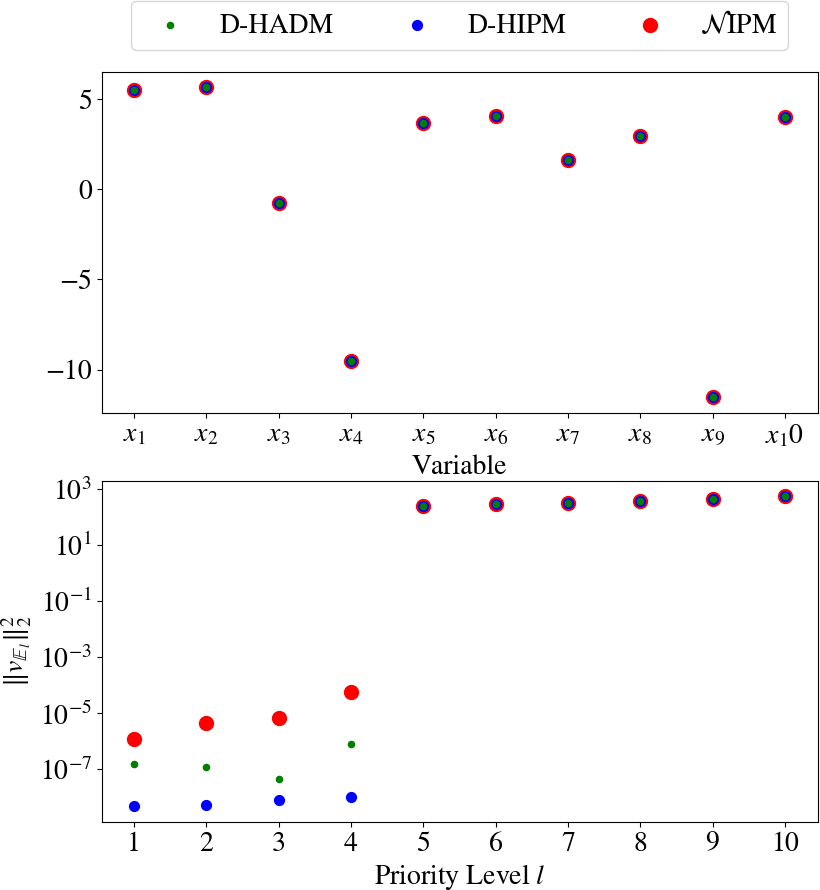}\hspace{5pt}
		\includegraphics[width=0.475
		\columnwidth]{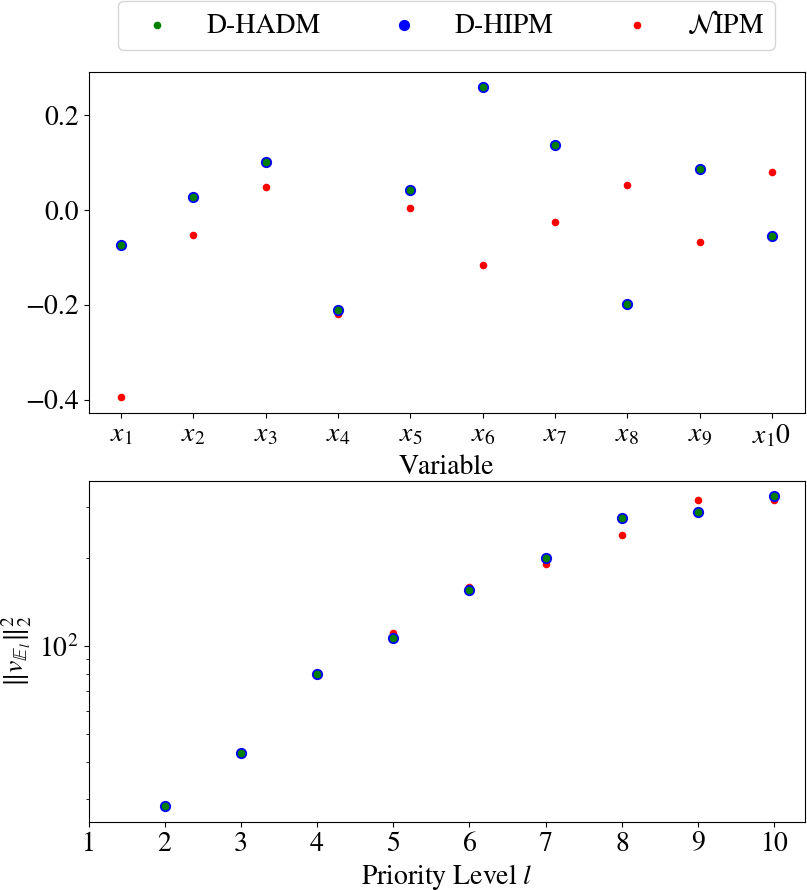}
		\centering
		\caption{Primal variables $x$ and norm of $v_{\bE_l}$ for an instance of a test with $p=n=10$, in the case of full rank (left) and rank-deficiency (right) of each constraint matrix $A_{\bE_l}$ (with $l=1,\dots,p$).}
		\label{fig:xv}
	\end{figure}
	
	Comparison of the behavior of the different HLSP solvers D-HADM (\ref{eq:d-hlsp} solver, based on ADMM), D-HIPM (\ref{eq:d-hlsp} solver, based on IPM), and $\mathcal{N}$IPM~\cite{pfeiffer2021} (\ref{eq:hlsp} solver, based on IPM) are given in Fig. \ref{fig:solverdata}. Each solver repeatedly (100 times) solves  hierarchies of $p=1,\dots,10$ levels with $n=p$ variables. Each level $l=1,\dots,p$ is composed of $l$ randomized equality constraints. On each level, 50\% of the constraints are linearly dependent with some slight perturbation (additive noise of magnitude $1\cdot10^{-12}$). The tests are run on an Intel i7-8550 CPU @ 1.80 GHz with 16 GB RAM.
	
	The proposed solver D-HADM shows a moderate increase in computation time with the number of priority levels and number of primal-dual variables. This is in contrast to D-HIPM, where the computational complexity grows cubically with the number of primal-dual variables $\vert \Lambda_{p-1}\vert$. For example, for $p=10$, D-HADM is about one magnitude faster than D-HIPM. For an instance of solving a problem with $p=9$, individual parts of the D-HADM aglorithm~\ref{alg:hadm} make up following parts of the overall computation time of 15~ms (with 700 solver iterations, no $\rho$ updates, $\Vert k_{prim,dual} \Vert_2^2 = 6.9\cdot 10^{-5}$):
	\begin{itemize}
		\itemsep0pt
		\item Computing and factorizing the substituted KKT matrix \eqref{eq:Kxrho} takes about 0.5\% of the time.
		\item Computing the right hand side \eqref{eq:kxrho} takes about 16\% of the computation time.
		\item Solving the substituted KKT system 	\eqref{eq:Kx} takes about 3\% of the computation time.
		\item Updating the primal-dual variables \eqref{eq:comppd} makes up about 11\% of the computation time.
		\item Solving \ref{eq:qcproj} takes about 23\%, while the actual root-finding of the third-order polynomial occupies about 5\%  of the time.
		\item Updating the dual variables of the dual problem (Sec. \ref{sec:dualascent}) requires about 21\% of the computation time.
	\end{itemize}
	The dual solvers are not competitive with the primal solver $\mathcal{N}$IPM, which is about 2-3 magnitudes faster. For one, the primal HLSP formulation is linear and therefore converges on each priority level after one iteration. In contrast, D-HADM and D-HIPM solve a non-linear QCLSP, which requires the iterative solution of linearized problems. Furthermore, by virtue of the rank-identifying nature of computing nullspace bases of active constraints of each level, $\mathcal{N}$IPM can  terminate as soon as the rank of the active constraints reaches maximum rank. For example, for $p=10$, $\mathcal{N}$IPM terminates after level 6 in some samples due to the linear dependency in the constraint matrices $A_{\bE_{\cup p}}$. 
	
	Convergence accuracy of D-HADM is moderate at levels of $1\cdot 10^{-2}$. This can partially be attributed to the incomplete pre-conditioning for computational efficiency described in Sec.~\ref{sec:precond}, but is also characteristic to first-order methods like ADMM~\cite{boyd2011}. On the other hand, the dual solver D-HIPM converges at similar accuracy as the primal solver $\mathcal{N}$IPM.
	
	The primal variables $x$ and norm of $v_{\bE_l}$ are depicted in Fig.~\ref{fig:xv} for one sample of a test with $p=n=10$. In the case of full rank of each constraint matrix $A_{\bC_{l}}$ ($l=1,\dots,p$), the primal variables $x$ and the primal objectives $\frac{1}{2}\Vert v_{\bE_l}\Vert_2^2$ of both the primal and dual solvers correspond. In the case of rank-deficiency, the primal variables $x$ are not uniquely determined and differ between the primal and dual solvers. The values of the primal objectives correspond up to level 4, but are more accurate for the dual solvers from level 5 with $\Vert v_{\bE_5}\Vert_2^2 = 106.7$, and $\Vert v_{\bE_5}\Vert_2^2 = 110.5$ for the primal solver. In this example, all solvers converge accurately at $\Vert k^{\text{D-HADM}} \Vert_2^2 = 7.9\cdot 10^{-5}$, $\Vert k^{\text{D-HIPM}} \Vert_2^2 = 2.7\cdot 10^{-4}$, and $\Vert k^{\mathcal{N}\text{IPM}} \Vert_2^2 = 6.8\cdot 10^{-7}$.

	\section{Conclusion}
	\label{sec:conclusion}
	In this work, an ADMM solver has been developed which solves the dual form of HLSP's with equality constraints. Some computational efficiency is achieved by substitution of primal-dual variables which do not need to be considered in expensive matrix factorizations. The solver is shown to be computationally more efficient than a dual HLSP solver based on the IPM. However, the solver is significantly slower than conventional primal HLSP solvers for example based on the IPM.
	
	Despite their computational challenges, dual HLSP solvers offer distinct advantages with respect to primal HLSP solvers. They are continuous and differentiable, which builds the foundation for a broader application of hierarchical programming with linear constraints in future work. Computation in a distributed fashion could contribute to more refined hierarchical control principles in robot swarm control. Inclusion of HLSP's as neurons in neural networks can help learning hierarchical patterns in the training data. Here, it needs to be investigated how the gradient computation in Sec. \ref{sec:diffdhlsp} can be designed efficiently. Especially, it remains to be seen whether the linear system \eqref{eq:lindiffsystem} can be solved efficiently, for example by re-using values from the ADMM or IPM based \ref{eq:d-hlsp} solvers. 
	
	The proposed dual HLSP solvers based on convex formulation is not able to consider inequality constraints due to the presence of non-convex complementary constraints. Future research is dedicated to identifying convex approximations to these constraints for computationally efficient handling of inequality constraints.

	\bibliographystyle{plain} 
	\bibliography{bib.bib}
	
	\appendix
	
	\section{Gradient of~\ref{eq:d-hlsp}}
	\label{sec:diffdhlsp}
	
	Constrained continuous programs like QP's and quadratically constrained quadratic programs (QCQP) are differentiable with respect to their problem parameters, as has been demonstrated in \cite{optnet} and \cite{Lidec2021}, respectively. The gradient can be formulated by taking differentials of the first-order optimality conditions of the Lagrangian \eqref{eq:d-hlsp-lag} of the \ref{eq:d-hlsp}. Unlike \cite{Lidec2021}, the gradient of \ref{eq:d-hlsp} additionally needs to consider linear equality constraints. The complementary KKT conditions of the quadratic inequality constraints in \ref{eq:d-hlsp} are (for $l = 1,\dots,p-1$)
	\begin{align}
	&\theta_{p,l}( {v}_{\bE_l}^T({v}_{\bE_l} + b_{\bE_l}) +  \lambda_{l,\bE_{\cup l-1}}^Tb_{\bE_{\cup l-1}}) = 0 \nonumber
	\end{align} 
	The differentials of the KKT condition of \ref{eq:d-hlsp} with respect to the primal and dual variables and the problem parameters $A_{\bE_{\cup p}}$ and $b_{\bE_{\cup p}}$ result in the linear system
	\begin{align}
	\dd K_{\psi}(x,v_{\bE_{\cup p}},\lambda_{\cup p, \bE_{\cup p}},\eta_{p,\cup p-1},\theta_{p,\cup p-1}) \dd \psi = \dd k_{\psi}(\dd A_{\bE_{\cup p}}, \dd b_{\bE_{\cup p}})
	\label{eq:lindiffsystem}
	\end{align}
	The vector $\dd q$ represents the Jacobians of the primal variables with respect to the parameters of the D-HLSP
	\begin{align}
	\dd\psi &= 		\BIN \dd x^T & \dots & \dd v_{\bE_l}^T & \dd \mu_{p,\bE_l}^T & \dd\theta_{p,l}^T & \dd \eta_{p,l}^T   & \dots & \dd v_{\bE_p}^T & \dd \mu_{p,\bE_p}^T & \dots & \dd \lambda_{l,\bE_{\cup l-1}}^T & \dots & \BOUT^T
	\end{align}
	The matrix $\dd K_{\psi}$ and vector $\dd k_{\psi}$ are defined as
	\begin{align}
	\dd K_{\psi} \coloneqq&
	\resizebox{1\hsize}{!}{$ \left[\begin{array}{@{}cccccccccccccccccccccc@{}}
		{\color{gray}\dd x} &
		{\color{gray}\dots} &
		{\color{gray}\dd v_{\bE_l}} & 
		{\color{gray}\dd \mu_{p,\bE_l}}&
		{\color{gray}\dd \theta_{p,l}}&
		{\color{gray}\dd \eta_{p,l}}&
		{\color{gray}\dots} &
		{\color{gray}\dd v_{\bE_p}} &
		{\color{gray}\dd \mu_{p,\bE_p}} &
		{\color{gray}\dots} &
		{\color{gray}\dd \lambda_{l,\bE_{\cup l-1}}} &
		{\color{gray}\dots}
		\\
		&\dots&& A_{\bE_{l}}^T &&&\dots&A_{\bE_{p}}^T&&\dots&&\dots \\
		\hline\\
		\vdots & \ddots & \vdots& \vdots& \vdots& \vdots& \ddots& \vdots& \vdots& \ddots& \vdots& \ddots\\
		\hline\\
		&\dots& 2\theta_{p,l}I & -I & 2v_{\bE_{l}} + b_{\bE_{l}} & 	A_{\bE_{l}} & \dots &&& \dots && \dots  \\
		A_{\bE_{l}} &\dots& -I &&&& \dots &&& \dots && \dots  \\
		&\dots& \theta_{p,l}(2v_{\bE_{l}} + b_{\bE_{l}})^T && {v}_{\bE_l}^T({v}_{\bE_l} + b_{\bE_l}) +  \lambda_{l,\bE_{\cup l-1}}^Tb_{\bE_{\cup l-1}}& &\dots&&&\dots& \theta_{p,l}b_{\bE_{\cup l-1}}^T & \dots\\ 
		&\dots& A_{\bE_{l}}^T &&&&\dots&&&\dots&  A_{\bE_{l-1}}^T & \dots\\
		\hline\\
		\vdots & \ddots & \vdots&\vdots&\vdots&\vdots& \ddots &\vdots&\vdots& \ddots &\vdots& \ddots &\\
		\hline\\
		&\dots&&&&&\dots&I&-I & \dots && \dots \\
		A_{\bE_{p}} &\dots&&&&&\dots& -I && \dots && \dots  \\
		\hline\\
		\vdots & \ddots & \vdots&\vdots&\vdots&\vdots& \ddots &\vdots&\vdots& \ddots &\vdots& \ddots &\\
		&\dots&&&  b_{\bE_{\cup l-1}} & A_{\bE_{\cup l-1}} & \dots &&& \dots & & \dots &\\
		\vdots & \ddots & \vdots&\vdots&\vdots&\vdots& \ddots &\vdots&\vdots& \ddots &\vdots& \ddots &\\
		\end{array}\right]$}
	\end{align}
	and
	\begin{align}
	\dd k_{\psi} \coloneqq &
	\BIN  - \sum_{l=1}^{p} \dd A_{\bE_l}^T\mu_{p,\bE_l}\\
	\vdots\\
	-\theta_{p,l}\dd b_{\bE_l} - \dd A_{\bE_l}\eta_{l} \\
	-\dd A_{\bE_l}x + \dd b_{\bE_l}\\
	-\theta_{p,l}({v}_{\bE_l}^T \dd b_{\bE_l} +  \lambda_{l,\bE_{\cup l-1}}^T\dd b_{\bE_{\cup l-1}}) \\
	- \dd A_{\bE_l}^Tv_{\bE_l} -  \dd A_{\bE_{\cup l-1}}^T\lambda_{l,\bE_{\cup l-1}}\\
	\vdots\\
	0\\
	-\dd A_{\bE_p}x + \dd b_{\bE_p}\\
	\vdots \\
	- \dd A_{\bE_{\cup l-1}}\eta_{p,l} + \theta_{p,l}\dd b_{\bE_{\cup l-1}} \\
	\vdots \BOUT
	\end{align}
	The Jacobian with respect to the desired parameter can be obtained by solving the linear system 	\eqref{eq:lindiffsystem} while the corresponding differential is set to the identity matrix, and the remaining ones are set to zero. This requires the inversion of the matrix $\dd K_{\psi}$. Future work is dedicated on how this can be achieved efficiently, possibly by reusing computations from the dual HLSP solver.

	\section{A primal-dual interior-point method for \ref{eq:d-hlsp}}
	\label{app:ipm}
	
	In order to motivate the choice of the ADMM as the solver-type of choice for \ref{eq:d-hlsp}, it is shown here how the IPM leads to an inefficient KKT system. Critically, primal-dual variables cannot be eliminated, which requires a factorization of a larger square matrix of dimension $n_x + \sum_{l=0}^{p-1} n_l^{dual}$ in every solver iteration, instead of $n_x$ as seen for the ADMM based solver presented in Sec. \ref{sec:hadmm}.
	
	First, the slack variable  $w_l^{\epsilon}$ is introduced in order to handle the dual constraints of each priority level.
	\begin{align}
	\mini_{{x},{v}_{\bE_{\cup p}},\Lambda_{p-1},w_{\bE_{\cup p}},w_{\cup p-1}}& \qquad \frac{1}{2}\Vert {v}_{\bE_p} \Vert^2_2 - \sigma\mu \sum_{l=1}^{p-1} \log(w_l^{\epsilon}) \label{eq:hlspdualipm}\\
	\text{s.t.}
	& \qquad A_{\bE_l}{x} - b_{\bE_l}  =  {v}_{\bE_l}\qquad l=1,\dots,p\nonumber\\
	&\qquad A_{\bE_l}^T{v}_{\bE_l} +  A_{\bE_{\cup l-1}}^T\lambda_{l,\bE_{\cup l-1}}= 0\nonumber\\
	& \qquad {v}_{\bE_l}^T({v}_{\bE_l} + b_{\bE_l}) +  \lambda_{l,\bE_{\cup l-1}}^T b_{\bE_{\cup l-1}} = w_l^{\epsilon}\nonumber
	\end{align}
	The Lagrangian of the above problem writes as
	\begin{align}
	\mathcal{L}_p^{\text{IPM}}(x,v,\lambda) &= \frac{1}{2}\Vert {v}_{\bE_p} \Vert^2_2 - \sigma\mu \sum_{l=1}^p\log(w_{l}^{\epsilon}) + \sum_{l=1}^{p-1} \mu_{p,\bE_l}^T(A_{\bE_l}{x} - b_{\bE_l} - {v}_{\bE_l}) \\
	&+ \sum_{l=1}^{p-1} \eta_{p,l}^T\left(A_{\bE_l}^T{v}_{\bE_l} + A_{\bE_{\cup l-1}}^T\lambda_{l,\bE_{\cup l-1}}\right) + \theta_{p,l}\left(v_{\bE_l}^T(v_{\bE_l} + b_{\bE_l})+\lambda_{l,\bE_{\cup l-1}}^Tb_{\bE_{\cup l-1}} + w_{l}^{\epsilon}\right)  \nonumber
	\end{align}
	The KKT conditions $K^{\text{IPM}}$ are composed of
	\begin{align}
	k_x^{\text{IPM}} \coloneqq \nabla_{{x}} \mathcal{L}_p^{\text{IPM}} =&  A_{\bE_{\cup p}}^T\mu_{p,\bE_{\cup p}} = 0\nonumber
	\end{align}
	and for $l=1,\dots,p-1$
	\begin{align}
	k_{v_{\bE_l}}^{\text{IPM}} &\coloneqq \nabla_{v_{\bE_l}} \mathcal{L}_p = -\mu_{p,\bE_l} + \theta_{p,l}(2{v}_{\bE_l} + b_{\bE_l}) + A_{\bE_1}\eta_{p,l}= 0\nonumber\\
	k_{\mu_{p,\bE_l}}^{\text{IPM}} &\coloneqq \nabla_{\mu_{p,\bE_l}} \mathcal{L}_p = A_{\bE_l}{x} - b_{\bE_l} - {v}_{\bE_l}  = 0\nonumber\\
	k_{\eta_{p,l}}^{\text{IPM}} &\coloneqq \nabla_{\eta_{p,l}} \mathcal{L}_p = A_{\bE_l}^T{v}_{\bE_l} +  A_{\bE_{\cup l-1}}^T\lambda_{l,\bE_{\cup l-1}} = 0\nonumber\\
	k_{\theta_{p,l}}^{\text{IPM}} &\coloneqq \nabla_{\theta_{p,l}} \mathcal{L}_p = {v}_{\bE_l}^T({v}_{\bE_l} + b_{\bE_l}) +\lambda_{l,\bE_{\cup l-1}}^Tb_{\bE_{\cup l-1}} + w_{l}^{\epsilon} = 0\nonumber\\
	k_{w_{l}^{\epsilon}}^{\text{IPM}} &\coloneqq \nabla_{w_{l}} \mathcal{L}_p = \theta_{p,l} \odot w_{l}^{\epsilon} - \sigma\mu= 0\nonumber
	\end{align}
	and for $l = 2,\dots,p-1$
	\begin{align}
	k_{\lambda_{l,\bE_{\cup l-1}}}^{\text{IPM}} &\coloneqq \nabla_{\lambda_{l,\bE_{\cup l-1}}} \mathcal{L}_p = \theta_{p,l} b_{\bE_{l-1}} + A_{\bE_{\cup l-1}}\eta_{p,l}= 0\nonumber
	\end{align}
	and for level $p$
	\begin{align}
	k_{v_{\bE_p}}^{\text{IPM}} &\coloneqq \nabla_{v_{\bE_p}} \mathcal{L}_p = v_{\bE_p} - \mu_{p,\bE_p} = 0\nonumber\\
	k_{\mu_{p,\bE_p}}^{\text{IPM}} &\coloneqq \nabla_{\lambda_{p,\bE_p}} \mathcal{L}_p = A_{\bE_p}{x} - b_{\bE_p} - {v}_{\bE_p} = 0\nonumber
	\end{align}
	Applying Newton's method to the non-linear KKT conditions results in
	\begin{align}
	K^{\text{IPM}} \Delta {\psi} = -k^{\text{IPM}}({\psi})
	\end{align}
	The variable vector $\psi$ is defined as (with $l = 1,\dots,p-1$ and $k = 1,\dots,l-1$)
	\begin{align}
	\psi = 		\BIN x^T & \dots & v_{\bE_l}^T & \mu_{p,\bE_l}^T & \eta_{p,l}^T & \theta_{p,l}^T  & w_l^{\epsilon,T} & \dots & v_{\bE_p}^T & \mu_{p,\bE_p} & \dots & \lambda_{l,\bE_{\cup l-1}}^T & \dots\BOUT^T\nonumber
	\end{align}
	The Lagrangian Hessian $K^{\text{IPM}} \coloneqq \nabla_{\psi} k^{\text{IPM}}$ is given by (with $l=1,\dots,p-1$)
	
	\begin{align}
	K^{\text{IPM}}_l \coloneqq
	\resizebox{0.85\hsize}{!}{$ \left[\begin{array}{@{}cccccccccccccccccccccc@{}}
		{\color{gray}\Delta x} &
		{\color{gray} \dots} &
		{\color{gray}\Delta v_{\bE_l}} & 
		{\color{gray}\Delta \mu_{p,\bE_l}}&
		{\color{gray}\Delta \theta_{l}}&
		{\color{gray}\Delta \eta_{l}}&
		{\color{gray}\Delta w_{l}^{\epsilon}}&
		{\color{gray}\dots} &
		{\color{gray}\Delta v_{\bE_p}} & {\color{gray}\Delta \mu_{p,\bE_p}} &
		{\color{gray}\dots} &
		{\color{gray}\Delta \lambda_{l,\bE_{\cup l-1}}} & 
		{\color{gray}\dots}
		\\
		&\dots&&& A_{\bE_{l}}^T &&&\dots&A_{\bE_{l}}^T&&\dots&&\dots \\
		\hline\\
		\vdots & \ddots & \vdots& \vdots& \vdots& \vdots& \vdots& \ddots& \vdots& \vdots& \ddots& \vdots& \ddots\\
		\hline\\
		&\dots& 2\theta_{p,l}I & -I & 2v_{\bE_{l}} + b_{\bE_{l}} & 	A_{\bE_{l}} && \dots &&& \dots && \dots  \\
		A_{\bE_{l}} &\dots&-I&&&&& \dots &&& \dots && \dots \\
		&\dots& (2v_{\bE_{l}} + b_{\bE_{l}})^T &&&& &\dots&&& \dots & b_{\bE_{\cup l-1}}^T & \dots\\ 
		&\dots& A_{\bE_{l}}^T &&&&&\dots&&&\dots & A_{\bE_{\cup l-1}}^T & \dots\\
		&\dots&&& \diag(w_{l}^{\epsilon}) && \diag(\theta_{p,l}) &\dots &&& \dots && \dots \\
		\hline\\
		\vdots & \ddots & \vdots& \vdots& \vdots& \vdots& \vdots& \ddots& \vdots& \vdots& \ddots& \vdots& \ddots\\
		\hline\\
		&\dots&&&&&&\dots&I&-I & \dots && \dots \\
		A_{\bE_{p}} &\dots&&&&&&\dots& -I && \dots && \dots  \\
		\hline\\
		\vdots & \ddots & \vdots& \vdots& \vdots& \vdots& \vdots& \ddots& \vdots& \vdots& \ddots& \vdots& \ddots\\
		&\dots&&& b_{\bE_{\cup l-1}} & A_{\bE_{\cup l-1}}&& \dots &&& \dots &0& \dots \\
		\vdots & \ddots & \vdots& \vdots& \vdots& \vdots& \vdots& \ddots& \vdots& \vdots& \ddots& \vdots& \ddots
		\end{array}\right]$}\nonumber
	\end{align}
	This system is repeatedly solved for $\Delta {\psi}$, and the step is applied by ${\psi}_{i+1} \leftarrow {\psi}_i + \alpha\Delta {\psi}_i$. The line-search factor $\alpha$ ensures dual feasibility 
	$\theta_{p,l} \geq 0$ 
	and  $w_l^{\epsilon}\geq 0$.
	The above system is sparse, and substitutions can be found for most primal and dual variables. However, it can be observed that the primal-dual variables can not be eliminated, as there are no invertible elements on the corresponding diagonal entries.
	Therefore, in each Newton iteration, a substituted KKT matrix of square dimensions $n_x + \sum_{l=2}^{p-1} n_{l}^{dual}$ needs to be factorized.

	\section{A primal-dual interior-point method for \ref{eq:qcproj}}
	\label{app:ipmepsilon}
	The splitting variable update \ref{eq:qcproj} can only be done efficiently if an incomplete scaling to the KKT matrix \eqref{eq:equPA} is applied (see Sec.~\ref{sec:precond}). This can carry negative implications on the overall ADMM convergence. In the following, an IPM free of matrix factorizations is proposed. Despite its iterative nature and the substantial amount of vector-vector operations, it is a valuable tool if the focus is on numerical stability at reasonable computational expenditure.
	
	\ref{eq:qcproj} is transformed into an equality-only constrained program by using log-barriers for the inequality constraint
	\begin{align}
	\mini_{z_l,\lambda_{l,\bE_{\cup l-1}}} \qquad& \frac{1}{2}\left\Vert 
	\BIN V_{\phi,l}z_{\bE_l} \\ V_{\nu,l}\tilde{\lambda}_{l,\bE_{\cup l-1}} \BOUT
	-
	\BIN
	\overline{a}_{1,l}\\
	\overline{a}_{2,\cup l-1}\\
	\BOUT
	\right\Vert_2^2 
	- \sigma\mu \log(w_l^{\epsilon}) \\
	s.t.\qquad & z_{\bE_l}^Tz_{\bE_l} - \hat{b}_{\bE_l}^T\hat{b}_{\bE_l} + \tilde{\lambda}_{l,\bE_{\cup l-1}}^Tb_{\bE_{\cup l-1}} = -w_l^{\epsilon}\nonumber
	\end{align}
	The Lagrangian of this problem writes as
	\begin{align}
	\mathcal{L}_l^{\epsilon} =&  \frac{1}{2}\left\Vert 
	\BIN V_{\phi,l}z_{\bE_l} \\ V_{\nu,l}\tilde{\lambda}_{l,\bE_{\cup l-1}} \BOUT
	-
	\BIN
	\overline{a}_{1,l}\\
	\overline{a}_{2,\cup l-1}\\
	\BOUT
	\right\Vert_2^2  - \sigma\mu \log(w_l^{\epsilon}) +
	\theta_{p,l} (z_{\bE_l}^Tz_{\bE_l} - \hat{b}_{\bE_l}^T\hat{b}_{\bE_l} + \tilde{\lambda}_{l,\bE_{\cup l-1}}^Tb_{\bE_{\cup l-1}} + w_l^{\epsilon}) 	
	\end{align}
	The first-order optimality conditions are
	\begin{align}
	k_l^{\epsilon}(\psi_l) = 
	\BIN 
	k_{z_{\bE_l}}^{\epsilon}\\
	k_{\tilde{\lambda}_{l,\bE_{\cup l-1}}}^{\epsilon}\\
	k_{w_l^{\epsilon}}^{\epsilon}\\
	k_{\theta_{p,l}}^{\epsilon}
	\BOUT
	\coloneqq 
	\BIN \nabla_{z_{\bE_l}}\mathcal{L}_l^{\epsilon} \\
	\nabla_{\tilde{\lambda}_{l,\bE_{\cup l-1}}}\mathcal{L}_l^{\epsilon}\\
	\nabla_{w_l^{\epsilon}}\mathcal{L}_l^{\epsilon}\\
	\nabla_{\theta_{p,l}}\mathcal{L}_l^{\epsilon}\BOUT
	=
	\BIN
	(V_{\phi,l}^2 + 2\theta_{p,l}I)z_{\bE_l} - V_{\phi,l}\overline{a}_{1,l} \\
	V_{\nu,l}^2\tilde{\lambda}_{l,\bE_{\cup l-1}} - V_{\nu,l}\overline{a}_{2,\cup l-1} + \theta_{p,l} b_{\bE_{\cup l-1}} \\
	\theta_{p,l} \odot w_l^{\epsilon} - \sigma\mu \\
	z_{\bE_l}^Tz_{\bE_l} - \hat{b}_{\bE_l}^T\hat{b}_{\bE_l} + \tilde{\lambda}_{l,\bE_{\cup l-1}}^Tb_{\bE_{\cup l-1}} + w_l
	\BOUT = 0
	\end{align}
	Applying Newton's method to the non-linear KKT conditions results in
	\begin{align}
	K_{l}^{\epsilon} \Delta {\psi}_l = -k_{l}^{\epsilon}({\psi_l}) \label{eq:epsilonnewton}
	\end{align}
	The variable vector $\psi_l$ is defined as 
	\begin{align}
	\psi_l \coloneqq 		\BIN z_{\bE_l}^T & \tilde{\lambda}_{l,\bE_{\cup l-1}}^T & w_l^{\epsilon}  & \theta_{p,l} \BOUT^T\nonumber
	\end{align}
	The Lagrangian Hessian $ K_l^{{\epsilon}} \coloneqq \nabla_{\psi} k_l^{\epsilon}$ is given by
	
	\begin{align}
	K^{\epsilon}_l \coloneqq
	{ \left[\begin{array}{@{}cccccccccccccccccccccc@{}}
		{\color{gray}\Delta z_{\bE_l}} & 	{\color{gray}\Delta \tilde{\lambda}_{l,\bE_{\cup l-1}}} &
		{\color{gray}\Delta w_l^{\epsilon}} &
		{\color{gray}\Delta \theta_{p,l}}\\
		(V_{\phi,l}^2 + 2\theta_{p,l}I) & & &2z_{\bE_l}\\
		& V_{\nu,l}^2 & &  b_{\bE_{\cup l-1}} \\
		& &  \theta_{p,l} &  w_l^{\epsilon}\\
		2z_{\bE_l}^T & b_{\bE_{\cup l-1}}^T & 1 
		\end{array}\right]}\nonumber
	\end{align}
	The system \eqref{eq:epsilonnewton} is solved iteratively until the KKT norm residual is below a given threshold $\Vert k_l^{\epsilon}\Vert_2 \leq \xi$.
	Substitutions in dependence of the scalar $\Delta \theta_{p,l}$ can be identified, such that a matrix factorization of $K_l^{{\epsilon}}$ is not necessary.
	\begin{align}
	\Delta z_{\bE_l} &= (V_{\phi,l}^2 + 2\theta_{p,l}I)^{-1}(-2 z_{\bE_l} \Delta \theta_{p,l}- k_{z_{\bE_l}}^{\epsilon})\nonumber\\
	\Delta w_l^{\epsilon} &= \frac{1}{\theta_{p,l}}(-w_l^{\epsilon}\Delta \theta_{p,l} - k_{w_l}^{\epsilon})\nonumber\\
	\Delta \tilde{\lambda}_{l,\bE_{\cup l-1}}  &= V_{\nu,l}^{-2}\left(- b_{\bE_{\cup l-1}} \Delta \theta_{p,l} -k_{\tilde{\lambda}_{l,\bE_{\cup l-1}}}^{\epsilon}\right)\nonumber
	\end{align}
	The above expressions are inserted into the last row of \eqref{eq:epsilonnewton}, which results in a formula for $\Delta \theta_{p,l}$
	\begin{align}
	&\Delta \theta_{p,l} = \frac{t_{2,l}}{t_{1,l}}\nonumber
	\end{align}
	with the scalars
	\begin{align}
	t_{1,l}&\coloneqq 4z_{\bE_l}^T(V_{\phi,l}^2 + 2\theta_{p,l}I)^{-1}z_{\bE_l} + b_{\bE_{\cup l-1}}^TV_{\nu,l}^{-2}b_{\bE_{\cup l-1}} + \frac{w_l^{\epsilon}}{\theta_{p,l}}\nonumber\\
	t_{2,l} &\coloneqq  -2z_{\bE_l}^T(V_{\phi,l}^2 + 2\theta_{p,l}I)^{-1}k_{z_{\bE_l}} - b_{\bE_{\cup l-1}}^TV_{\nu,l}^{-2}k_{\tilde{\lambda}_{l,\bE_{\cup l-1}}}^{\epsilon} - \frac{1}{\theta_{p,l}}k_{w_l^{\epsilon}}+ k_{\theta_{p,l}}\nonumber
	\end{align}
	Once the step $\Delta \psi$ has been identified, line-search $\psi \leftarrow \psi + \alpha \Delta \psi$ with $\alpha \in [0,1]$ ensures that the non-negativity conditions on $w_l^{\epsilon}\geq 0$ and $\theta_{p,l} \geq 0$ are respected.

\end{document}